\documentclass[11pt, reqno]{amsart}
\usepackage{amssymb, amsthm, amsmath, amsfonts,mathtools}
\usepackage{array, epsfig}
\usepackage{bbm}
\usepackage{enumitem}

\usepackage{hyperref}
\usepackage[numbers,square]{natbib}
\usepackage{color}
\allowdisplaybreaks

\numberwithin{equation}{section}

  \evensidemargin
\oddsidemargin
\newcommand{\stirling}[2]{\genfrac{[}{]}{0pt}{}{#1}{#2}}
\newcommand{\stirlingsec}[2]{\genfrac{\{}{\}}{0pt}{}{#1}{#2}}

\newcommand{\stirlingb}[2]{B\hspace*{-0.2mm}\big[{#1},{#2}\big]}
\newcommand{\stirlingsecb}[2]{B\hspace*{-0.2mm}\big\{{#1},{#2}\big\}}

\newcommand{\R}{\mathbb{R}}

\newcommand{\SP}{\mathbb{S}}

\newcommand{\F}{\mathcal{F}}
\newcommand{\1}{\mathbbm{1}}

\newcommand{\PP}{\mathbb{P}}

\newcommand{\D}{\mathcal{D}}
\newcommand{\eps}{\varepsilon}
\newcommand{\Dna}{\mathcal{D}_n^A}
\newcommand{\Dnb}{\mathcal{D}_n^B}

\newcommand{\A}{\mathcal{A}}

\newcommand*\xbar[1]{%
   \hbox{%
     \vbox{%
       \hrule height 0.5pt 
       \kern0.25ex
       \hbox{%
         \kern-0.05em
         \ensuremath{#1}%
         \kern-0.1em
       }%
     }%
   }%
}

\DeclareMathOperator{\E}{\mathbb{E}}

\DeclareMathOperator{\lin}{lin}

\DeclareMathOperator{\pos}{pos}

\DeclareMathOperator{\relint}{relint}

\DeclareMathOperator{\linsp}{linsp}

\newcommand{\cA}{\mathcal{A}}
\newcommand{\cB}{\mathcal{B}}

\newcommand{\cF}{\mathcal{F}}

\newcommand{\cW}{\mathcal{W}}

\DeclareMathOperator*{\Ker}{Ker}

\renewcommand{\P}{\mathbb{P}}

\newcommand{\aff}{\mathop{\mathrm{aff}}\nolimits}

\newcommand{\eqdistr}{\stackrel{d}{=}}

\usepackage{cleveref}
\theoremstyle{plain}
\newtheorem{satz}{Theorem}[section]
\newtheorem{lem}[satz]{Lemma}
\newtheorem{kor}[satz]{Corollary}
\newtheorem{prop}[satz]{Proposition}

\theoremstyle{definition}

\newtheorem{bsp}[satz]{Example}

\theoremstyle{remark}
\newtheorem{bem}[satz]{Remark}
\newtheorem*{bem*}{Remark}

\begin{document}

\author{Thomas Godland}
\address{Institut f\"ur Mathematische Stochastik,
Westf\"alische Wilhelms-Universit\"at M\"unster,
Orl\'eans-Ring 10,
48149 M\"unster, Germany}
\email{t\_godl01@uni-muenster.de}

\author{Zakhar Kabluchko}
\address{Institut f\"ur Mathematische Stochastik,
Westf\"alische Wilhelms-Universit\"at M\"unster,
Orl\'eans-Ring 10,
48149 M\"unster, Germany}
\email{zakhar.kabluchko@uni-muenster.de}

\title[Conical tessellations associated with Weyl chambers]{Conical tessellations associated with Weyl chambers}



\begin{abstract}
We consider $d$-dimensional random vectors $Y_1,\ldots,Y_n$ that satisfy a mild general position assumption a.s.
The hyperplanes
\begin{align*}
(Y_i-Y_j)^\perp\;\; (1\le i<j\le n),\quad
\end{align*}
generate a conical tessellation of the Euclidean $d$-space which is closely related to the Weyl chambers of type $A_{n-1}$. We  determine the number of cones in this tessellation and show that it is a.s.\ constant. For a random cone chosen uniformly at random from this random tessellation, we compute expectations of several geometric functionals. These include the face numbers, as well as the  conic intrinsic volumes and the conical quermassintegrals. Under the additional assumption of exchangeability on $Y_1,\ldots,Y_n$, the same is done for the dual random cones which have the same distribution as the positive hull of $Y_1-Y_2,\ldots, Y_{n-1}-Y_n$ conditioned on the event that this positive hull is not equal to $\mathbb R^d$. All these expectations turn out to be distribution-free. \\
Similarly, we consider the conical tessellation induced by the hyperplanes
\begin{align*}
(Y_i+Y_j)^\perp\;\; (1 \le i<j\le n),\quad
(Y_i-Y_j)^\perp\;\; (1\le i<j\le n),\quad
Y_i^\perp\;\; (1\le i\le n).
\end{align*}
This tessellation is closely related to the Weyl chambers of type $B_n$. We compute the number of cones in this tessellation and the expectations of various geometric functionals for random cones drawn from this random tessellation.\\
The main ingredient in the proofs is a connection between the number of faces of the tessellation and the number of faces of the Weyl chambers of the corresponding type that are intersected non-trivially by a certain linear subspace in general position.
\end{abstract}

\maketitle
\section{Main results}\label{sec:main_results}

\subsection{Introduction}\label{sec:introduction}
A \textit{polyhedral cone}
(or, for the purpose of the present paper, just a \textit{cone}) is an intersection of a finite number of
closed half-spaces whose boundaries pass through the origin.
Consider a linear \textit{hyperplane arrangement}, that is a finite collection
$\cA= \{H_1,\ldots,H_n\}$ of distinct hyperplanes in $\R^d$ passing through the origin.
These hyperplanes  dissect $\R^d$ into finitely many  polyhedral cones.
More precisely, the set $\R^d\setminus \bigcup_{i=1}^nH_i$
consists of finitely many open connected components whose closures define polyhedral cones.
The collection of these cones is called the \textit{conical tessellation} generated by $H_1,\ldots,H_n$.
Under the condition that the hyperplanes are in \textit{general position}, meaning that
\begin{align*}
\dim(H_{i_1}\cap\ldots \cap H_{i_k})=d-k
\end{align*}
for all $k\le d$ and all indices $1\le i_1<\ldots<i_k\le n$,  Schl\"afli~\cite{Schlaefli_buch_1950}
derived the following classical formula for the number $C(n,d)$ of cones generated by these hyperplanes:
\begin{align}\label{Eq_Number_Schlaefli_Cones}
C(n,d)=2\sum_{i=0}^{d-1}\binom{n-1}{i}.
\end{align}
For a simple inductive proof of this formula, see \cite[Lemma 8.2.1]{Schneider2008}.

If the hyperplanes  $H_1,\ldots,H_n$ are chosen at random, for example independently and uniformly in the space of all linear hyperplanes, we obtain a \textit{random} conical tessellation. By intersecting the cones of a conical tessellation with the unit sphere $\SP^{d-1}$  we obtain a tessellation of the unit sphere by spherical polytopes; see Figure~\ref{pic:1} for a sample realization in dimension $d=3$.  This tessellation has been studied by Cover and Efron~\cite{CoverEfron_paper} and  Hug and Schneider~\cite{HugSchneider2016}. For further results on this and other types of random tessellations of the sphere we refer to~\cite{miles,arbeiter_zaehle,barany_etal,schneider_intersection,kabluchko_thaele_vor,hug_thaele,cones_halfsphere,kabluchko_thaele_conv}.

\begin{figure}[!ht]
\centering
\includegraphics[scale=0.28]{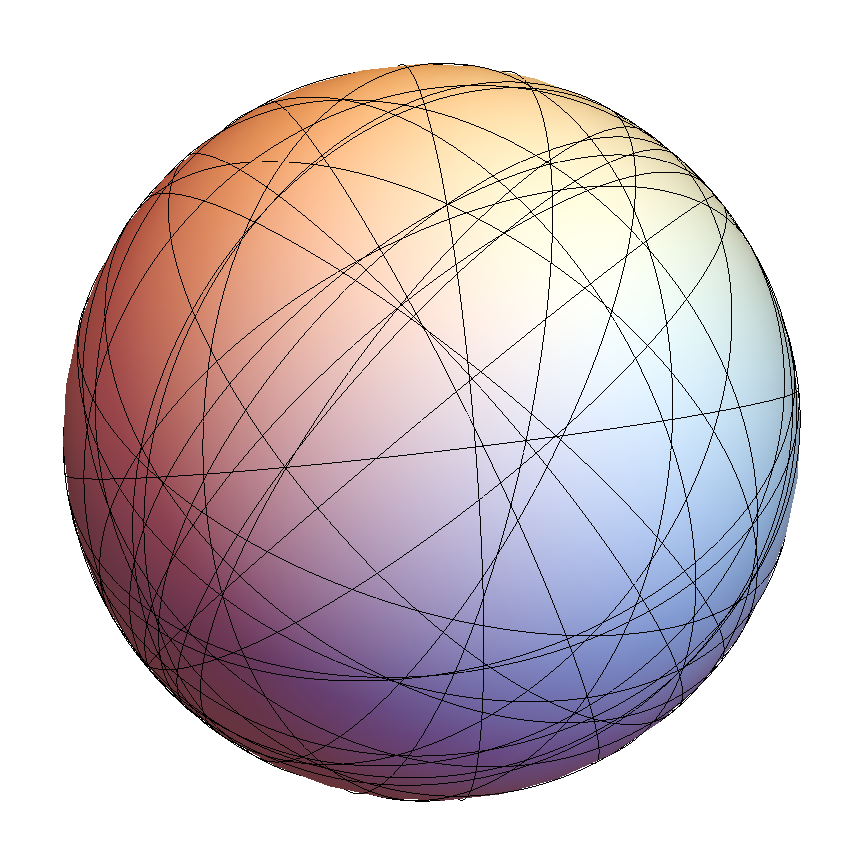}
\caption{Tessellation of the unit sphere in $\R^3$ induced by $n=36$ uniform and independent random hyperplanes.}\label{pic:1}
\end{figure}

Schl\"afli's formula, as well as related probabilistic results~\cite{Wendel_paper,CoverEfron_paper,HugSchneider2016},
are closely connected to the following question: How many orthants generated by the hyperplane arrangement consisting of the coordinate hyperplanes $e_1^\perp, \ldots, e_n^\perp$ in $\R^n$ are intersected by a linear subspace in general position? Here, $x^\perp=\{y\in\R^n:\langle x,y\rangle=0\}$ denotes the orthogonal complement of a vector $x\in\R^n\setminus\{0\}$, $\langle\cdot\,,\cdot\rangle$ denotes the standard Euclidean scalar product, while $e_1,\ldots,e_n$ is the standard orthonormal basis of $\R^n$. As it turned out~\cite{KVZ15,KVZ17,Kabluchko2019},
there are meaningful probabilistic problems (including, for example, the classical Sparre Andersen theorems on random walks)
that  are equivalent to a similar question for the so-called reflection arrangements of types $A_{n-1}$ and $B_{n}$, which are defined as
\begin{align}
\cA(A_{n-1})
&= \{(e_i-e_j)^\perp:\quad 1\le i<j\le n\},\label{eq:refl_arr_type_A}\\
\cA(B_n)
&=
\{
(e_i+e_j)^\perp,
(e_i-e_j)^\perp: 1\le i<j\le n\}
\cup
\{e_i^\perp: 1\le i\le n\}. \label{eq:refl_arr_type_B}
\end{align}

In this paper we want to introduce two new classes of conical tessellations that are related to these hyperplane arrangements
and the corresponding Weyl chambers. Let us start with tessellations of type $A_{n-1}$.

\subsection{Weyl tessellations of type \texorpdfstring{$\boldsymbol{A_{n-1}}$}{A\_n-1}: Number of cones and faces}\label{sec:weyl_tess_A}
Take some vectors $y_1,\ldots,y_n\in\R^d$.  By definition, the hyperplane arrangement $\A^A(y_1,\ldots,y_n)$ consists of the hyperplanes given by
\begin{align}\label{eq:A_arr_def}
(y_i-y_j)^\perp,\quad 1\le i<j\le n.
\end{align}
The \textit{Weyl tessellation of type $A_{n-1}$}, denoted by $\mathcal{W}^A(y_1,\ldots,y_n)$,  is defined as the conical tessellation generated by the hyperplane arrangement $\A^A(y_1,\ldots,y_n)$. An example is shown on the left panel of Figure~\ref{pic:2}.
When considering Weyl tessellations of type $A_{n-1}$, we always assume that the following condition holds:
\begin{enumerate}[label=(A\arabic*), leftmargin=50pt]
\item  For every permutation $\sigma$ of the set $\{1,\ldots,n\}$, any $d$ of the vectors $y_{\sigma(1)}-y_{\sigma(2)},\ldots,y_{\sigma(n-1)}-y_{\sigma(n)}$ are linearly independent, and $n\geq d+1$.
    \label{label_GP1_A}
\end{enumerate}
It is easy to see that this condition guarantees that the hyperplanes~\eqref{eq:A_arr_def} are well defined and pairwise distinct for $d\geq 2$, which we always assume in the following. Moreover, under~\ref{label_GP1_A}, the elements of $\mathcal{W}^A(y_1,\ldots,y_n)$ are exactly the cones of the form
$$
D_\sigma^A := \{v\in\R^d:\langle v, y_{\sigma(1)}\rangle\le \ldots\le \langle v,y_{\sigma(n)}\rangle\}
$$
that are not equal to $\{0\}$, where $\sigma$ runs over $\mathcal S_n$, the set of all permutations of $\{1,\ldots,n\}$. This is an easy consequence of~\cite[Eq.~(14)]{HugSchneider2016}.
In the next theorem we evaluate the number of cones in the Weyl tessellation of type $A_{n-1}$.

\begin{satz}\label{theorem:number_cones_A}
Let $y_1,\ldots,y_n\in\R^d$ satisfy the assumption~\ref{label_GP1_A}.
Then the number of cones in the Weyl tessellation $\mathcal{W}^A(y_1,\ldots,y_n)$ of type $A_{n-1}$ equals
\begin{align*}
D^A(n,d):=2\bigg(\stirling{n}{n-d+1}+\stirling{n}{n-d+3}+\ldots\bigg),
\end{align*}
where the $\stirling{n}{k}$'s are the Stirling numbers of first kind defined by the formula
\begin{align}\label{eq:def_stirling1}
t(t+1)\cdot\ldots\cdot(t+n-1)=\sum_{k=1}^n\stirling{n}{k}t^k
\end{align}
and, by convention, $\stirling{n}{k}=0$ for $k\notin\{1,\ldots, n\}$.
\end{satz}

The first thing to note is that the number of cones in the Weyl tessellation, like in the Schl\"afli case, does not depend on the choice of vectors $y_1,\ldots,y_n$, provided~\ref{label_GP1_A} holds. In Lemma~\ref{Lemma_AS_General_Position}
we will show that, in certain natural random settings, this condition is satisfied with probability $1$.
Moreover, in Theorem~\ref{Theorem_Aequivalenz_A1_A2} we will state an equivalent assumption,
called~\ref{label_GP2_A}, which allows to view~\ref{label_GP1_A} from a larger perspective of general position with respect to hyperplane arrangements.


For a polyhedral cone $C\subset \R^d$ and for
$k =0,\ldots, d$, denote by $\F_k(C)$ the set of all $k$-dimensional faces (or just $k$-faces) of $C$, and let $f_k(C):= \#\F_k(C)$ be their number.  The set of all $k$-dimensional faces  of the conic tessellation $\mathcal{W}^A(y_1,\ldots,y_n)$ is denoted by
\begin{align*}
\F^A_k(y_1,\ldots,y_n)=\bigcup_{C\in\mathcal{W}^A(y_1,\ldots,y_n)}\F_k(C).
\end{align*}
The next theorem states an explicit formula for the total number of $k$-faces in the Weyl tessellation of type $A_{n-1}$
and reduces to Theorem~\ref{theorem:number_cones_A} in the special case $k=d$.

\begin{satz}\label{Theorem_Number_WeylFaces_Tessellation_An-1}
Let $y_1,\ldots,y_n\in\R^d$ satisfy assumption~\ref{label_GP1_A}. Then the number of $k$-faces in the Weyl tessellation of type $A_{n-1}$ is given by
\begin{align*}
\#\F_k^A(y_1,\ldots,y_n)=\stirlingsec{n}{n-d+k}D^A(n-d+k,k),
\end{align*}
for all $k = 1,\ldots, d$. Here, $\stirlingsec{n}{k}$ is the Stirling number of the second kind, that is, the number of partitions of the set $\{1,\ldots,n\}$  into $k$ non-empty subsets.
\end{satz}

\subsection{Weyl tessellations of type \texorpdfstring{$\boldsymbol{B_n}$}{B\_n}: Number of cones and faces}\label{sec:weyl_tess_B}
We are now going to define conical tessellations of type $B_{n}$.
Take vectors $y_1,\ldots,y_n\in\R^d$.
By definition, the hyperplane arrangement $\mathcal{A}^B(y_1,\ldots,y_n)$ is the finite collection of hyperplanes in $\R^d$ given by
\begin{align}\label{eq:Weyl_arrangement_B}
(y_i+y_j)^\perp,\quad &1 \le i<j\le n,\nonumber\\
(y_i-y_j)^\perp,\quad &1\le i<j\le n,\\
y_i^\perp,\quad &1\le i\le n.\nonumber
\end{align}
Then, the \textit{Weyl tessellation of type $B_n$}, denoted by $\mathcal{W}^B(y_1,\ldots,y_n)$, is defined as the conical tessellation generated by the hyperplanes from $\mathcal{A}^B(y_1,\ldots,y_n)$. An example is shown on the right panel of Figure~\ref{pic:2}. When considering Weyl tessellations of type $B_n$ we always impose the following condition on the vectors $y_1,\ldots,y_n$:
\begin{enumerate}[label=(B\arabic*), leftmargin=50pt]
\item For every $\eps=(\eps_1,\ldots,\eps_n)\in\{\pm 1\}^n$ and every permutation $\sigma$ of the set $\{1,\ldots,n\}$, any $d$ of the vectors $\varepsilon_1y_{\sigma(1)}-\varepsilon_2y_{\sigma(2)},\varepsilon_2y_{\sigma(2)}-\varepsilon_3y_{\sigma(3)},\ldots,\varepsilon_{n-1}y_{\sigma(n-1)}-\varepsilon_ny_{\sigma(n)}, \varepsilon_ny_{\sigma(n)}$ are linearly independent, and $n\geq d$.
    \label{label_GP1}
\end{enumerate}


\begin{figure}[!ht]
\centering
\includegraphics[scale=0.33]{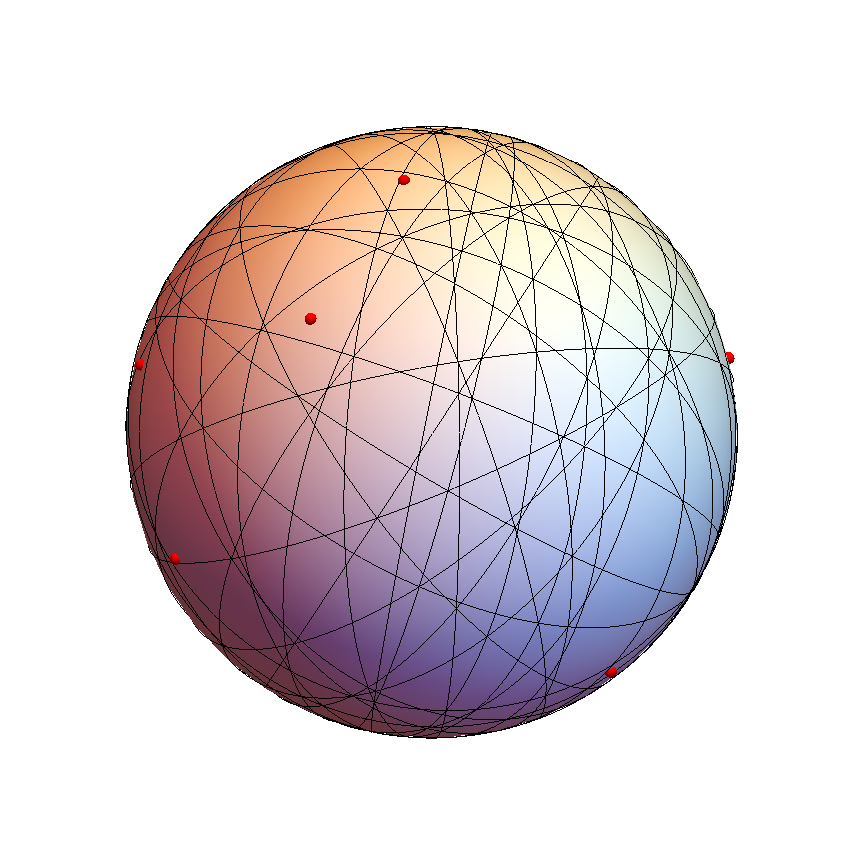}
\includegraphics[scale=0.33]{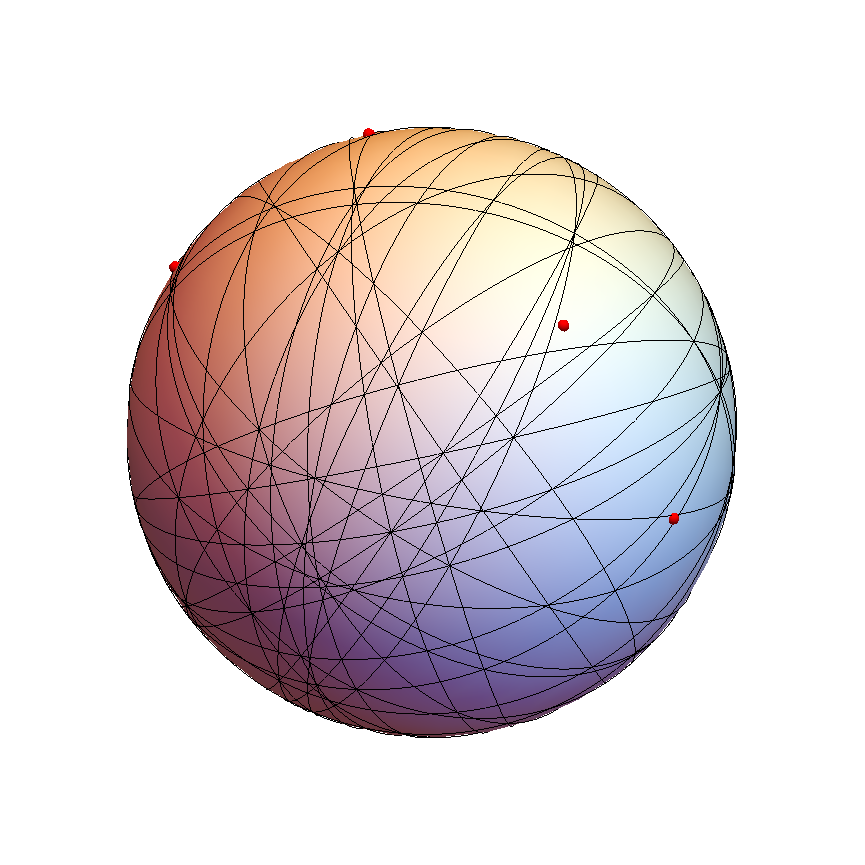}
\caption{Left: Weyl tessellation of type $A_{n-1}$ of the unit sphere in $\R^3$ with $n=9$.
Right: Weyl tessellation of type $B_n$ of the unit sphere in $\R^3$ with $n=6$.
Both  tessellations are generated by $36$ hyperplanes. The vectors $y_1,\ldots,y_n$ (red points) were sampled independently and uniformly on the unit sphere.
}\label{pic:2}
\end{figure}


It is easy to check that under~\ref{label_GP1} and for $d\geq 2$, which we tacitly assume in the following, the hyperplanes from
$\mathcal{A}^B(y_1,\ldots,y_n)$ are well-defined and pairwise distinct.
Moreover, under~\ref{label_GP1}, the elements of $\mathcal{W}^B(y_1,\ldots,y_n)$ are exactly those
 cones  of the form
\begin{align*}
D^B_{\varepsilon,\,\sigma}
:=\{v\in\R^d:\langle v,\varepsilon_1y_{\sigma(1)}\rangle\le \ldots\le \langle v,\varepsilon_ny_{\sigma(n)}\rangle\le 0\}
\end{align*}
that are different from $\{0\}$; see~\cite[Eq.~(14)]{HugSchneider2016}. Here,
$\varepsilon=(\varepsilon_1,\ldots,\varepsilon_n)\in\{\pm 1\}^n$ is a vector of signs and $\sigma\in \mathcal S_n$ runs through all permutations of $\{1,\ldots,n\}$.
\begin{satz}\label{Theorem_Number_Weyl_Cones}
Let $y_1,\ldots,y_n\in\R^d$  satisfy the assumption~\ref{label_GP1}. Then the number of cones in the Weyl tessellation $\mathcal{W}^B(y_1,\ldots,y_n)$ equals
\begin{align*}
D^B(n,d):=2\big(\stirlingb n{n-d+1}+\stirlingb n{n-d+3}+\ldots\big),
\end{align*}
where the numbers $B[n,k]$ are the coefficients of the polynomial
\begin{align}\label{eq:def_stirling1b}
(t+1)(t+3)\cdot\ldots\cdot(t+2n-1)=\sum_{k=0}^n\stirlingb nk t^k
\end{align}
and, by convention, $\stirlingb nk=0$ for $k\notin\{0,\ldots, n\}$.
\end{satz}

Similarly to the $A_{n-1}$-case, we denote the set of $k$-faces of $\mathcal{W}^B(y_1,\ldots,y_n)$
by $\cF^B_k(y_1,\ldots,y_n)$.  The next result generalizes Theorem~\ref{Theorem_Number_Weyl_Cones}.

\begin{satz}\label{Theorem_Number_WeylFaces_Tessellation_Bn}
Let $y_1,\ldots,y_n\in\R^d$ satisfy assumption~\ref{label_GP1}. Then the number of $k$-faces in the Weyl tessellation of type $B_n$ is given by
\begin{align*}
\#\F_k^B(y_1,\ldots,y_n)&=\stirlingsecb n{n-d+k}D^B(n-d+k,k)
\end{align*}
for all $k = 1,\ldots,d$, where the $\stirlingsecb n{k}$'s are given by
\begin{align}\label{eq:def_stirling2b}
\stirlingsecb n{k} = \sum_{r=k}^{n}\binom{n}{r}\stirlingsec{r}{k}2^{r-k}.
\end{align}
\end{satz}


The numbers $\stirlingb nk$ are known as the $B$-analogues of the Stirling numbers of the first
kind $\stirling{n}{k}$; see  entry A028338 in~\cite{sloane}. The numbers $\stirling{n}{k}$, respectively
 $\stirlingb{n}{k}$, appear in our formulas because they are (up to a sign) the
coefficients of the characteristic polynomials of the reflection arrangements of types $A_{n-1}$ and $B_{n}$ defined in~\eqref{eq:refl_arr_type_A} and~\eqref{eq:refl_arr_type_B}.
Our problems will be reduced
to certain questions on counting the number of faces and chambers of
reflection arrangements that are intersected by a linear subspace.
These questions can be answered in terms of the coefficients of the characteristic polynomials of the hyperplane arrangements, which can also be interpreted as Hilbert-Poincar\'e polynomials; see~\cite[Section~3.9]{humphreys_book}. We will directly rely on the corresponding results from~\cite{KVZ17} and ~\cite{KVZ15}, without using characteristic polynomials here.

Similarly, the numbers $\stirlingsecb nk$ are known~\cite{suter} as the $B$-analogues of the
Stirling numbers $\stirlingsec{n}{k}$ of the second kind; see Entry A039755 in~\cite{sloane}.
The  number of linear subspaces of dimension $k$
in the lattice generated by the reflection arrangement of type $A_{n-1}$, respectively, $B_n$,
is given by $\stirlingsec{n}{k}$, respectively $\stirlingsecb{n}{k}$.  On the other hand, the numbers
$\stirling{n}{k}$, respectively $\stirlingb{n}{k}$, count elements of the reflection groups of the corresponding type whose invariant subspace has dimension $k$.    Equivalently, they count permutations (respectively signed permutations) having exactly $k$ cycles
(respectively, exactly $k$ cycles with an even number of sign changes).
The $B$-analogues of the Stirling numbers of both kinds and some more general concepts appeared, for example, in~\cite{Amelunxen2017,bagno_biagioli_garber_some_identities,bagno_garber_balls,bala_stirling,dowling,drton_klivans,
henze_orakel,KVZ17,KVZ15,lang_stirling,suter}.



\subsection{Schl\"afli and Cover-Efron random cones}\label{sec:schlaefli_coverefron}
So far our results have been purely deterministic.
Hug and Schneider~\cite{HugSchneider2016} (who continued the work of Cover and Efron~\cite{CoverEfron_paper}) defined two natural families of \textit{random} cones as follows.  Let $X_1,\ldots,X_n$ be independent and identically distributed random vectors in $\R^d$  whose distribution is symmetric with respect to the origin and assigns probability $0$ to each linear hyperplane. The orthogonal complements of these vectors are the hyperplanes denoted by $H_1:= X_1^\bot,\ldots,H_n:=X_n^\bot$. Then, the \textit{Schl\"afli random cone} $S_n$ is the random cone obtained by picking uniformly at random one of the cones generated by the hyperplanes $H_1,\ldots,H_n$.  Furthermore, the \textit{Cover-Efron random cone} $C_n$ is defined as the positive hull of $X_1,\ldots,X_n$ in $\R^d$, that is
\begin{align*}
\pos \{X_1,\ldots,X_n\}=\left\{\sum_{i=1}^n\lambda_iX_i: \lambda_1,\ldots,\lambda_n\ge 0\right\},
\end{align*}
conditioned on the event that this positive hull is not equal to $\R^d$.
It has been shown in~\cite{HugSchneider2016}  that, in distribution, the Cover-Efron cone $C_n$ coincides with the dual cone of $S_n$, where the dual cone of a cone $C\subset \R^d$ is defined by
\begin{align*}
C^\circ:=\{v\in\R^d:\langle x,v\rangle\le 0\;\text{ for all } x\in C\}.
\end{align*}

Hug and Schneider~\cite{HugSchneider2016} evaluated expectations of a few geometric functionals   of $S_n$ and $C_n$.  We will recall their results on $S_n$ because the results on $C_n$ follow by duality.  The expected number of $j$-faces of $S_n$ is given by
\begin{align*}
\E f_j(S_n)=\frac{2^{d-j}\binom{n}{d-j}C(n-d+j,j)}{C(n,d)},
\end{align*}
for all $j = 1,\ldots,d$.
Moreover, Hug and Schneider~\cite{HugSchneider2016} generalized this result  by introducing a series of general geometric functionals $Y_{k,j}$, called the \textit{size functionals}. In order to define them, we need to recall the definition of the conical quermassintegrals.  For a cone $C\subseteq \R^d$ that is not a linear subspace (which is the only case we will face in this paper) the \textit{conical quermassintegrals} are defined by
\begin{align}\label{Eq_Quermassintegrals_No_Subspace}
U_j(C)=\frac{1}{2}\int_{G(d,d-j)}\1_{\{C\cap L\neq\{0\}\}}\,\nu_{d-j}(\text{d}L),\quad j=0,\ldots,d,
\end{align}
where the Grassmannian of all $j$-dimensional linear subspaces in $\R^d$ is denoted by $G(d,j)$,
and $\nu_j$ is the unique rotationally invariant Borel probability measure on $G(d,j)$.
Thus, $2U_j(C)$ is the probability that a uniformly distributed $(d-j)$-dimensional random linear subspace intersects $C$ non-trivially.
Then, following~\cite{HugSchneider2016}, the functional $Y_{k,j}(C)$ is defined as
\begin{align*}
Y_{k,j}(C):=\sum_{F\in\cF_k(C)}U_j(F), \quad 0\le j<k\le d.
\end{align*}
The quantities $Y_{k,j}$ are significant since they comprise a lot of important geometric functionals such as the number of $k$-faces of $C$, which is given by  $f_k(C) = 2 Y_{k,0}(C)$  (provided no $k$-face is a linear subspace), and the conical quermassintegrals $U_j(C)= Y_{\dim C, j}(C)$.
In \cite[Theorem 4.1]{HugSchneider2016}, Hug and Schneider derived a formula for the expected size functionals of $S_n$, namely
\begin{align}\label{Eq_Expected_SizeFunct_Sn}
\E Y_{d-k+j,\,d-k}(S_n)=\frac{2^{k-j}\binom{n}{k-j}C(n-k+j,j)}{2C(n,d)}
\end{align}
for all $1\le j\le k\le d$ and $n>k-j$.

\subsection{Weyl random cones}\label{sec:weyl_random_cones}
A natural question arises whether similar calculations are possible for a cone randomly chosen from the Weyl tessellations.
At first, we consider the type $A_{n-1}$. Let $Y_1,\ldots,Y_n$  be (possibly dependent) random vectors in $\R^d$ satisfying assumption~\ref{label_GP1_A} a.s.
For example, if the tuple $(Y_1,\ldots,Y_n)$ has a
joint density with respect to the Lebesgue measure on $(\R^d)^n$ or the product spherical Lebesgue measure on
$(\SP^{d-1})^n$, then \ref{label_GP1_A} is satisfied a.s.,\ as will be shown
in Lemma~\ref{Lemma_AS_General_Position}.
The \textit{Weyl random cone} $\Dna$ of type $A_{n-1}$ is defined as follows: Among the cones of the random Weyl tessellation $\mathcal{W}^A(Y_1,\ldots,Y_n)$ choose one uniformly at random.
The next two theorems give formulas for the expected number of $k$-faces of $\Dna$ and, more generally,
for the expected size functionals $Y_{k,\,j}$.

\begin{satz}\label{theorem:exp_k-faces_Weyl_cones_A}
Let $Y_1,\ldots,Y_n$ be random vectors in $\R^d$ that satisfy \ref{label_GP1_A} a.s.
Then, for $k=1,\ldots,d$, the expected number of $k$-faces of the Weyl random cone $\Dna$ of type $A_{n-1}$ is given by
\begin{align*}
\E f_k(\Dna)=\frac{\tbinom{n-1}{d-k}D^A(n-d+k,k)}{D^A(n,d)}\frac{n!}{(n-d+k)!}.
\end{align*}
\end{satz}
\begin{satz} \label{Theorem_sizefunction_Dn_An-1}
Let $Y_1,\ldots,Y_n$ be random vectors in $\R^d$ that satisfy \ref{label_GP1_A} a.s. and let $\Dna$
be a Weyl random cone of type $A_{n-1}$. Then, for all $1\le j\le k\le d$ it holds that
\begin{align*}
\E Y_{d-k+j,\,d-k}(\Dna)=\frac{\tbinom{n-1}{k-j}D^A(n-k+j,j)}{2D^A(n,d)}\frac{n!}{(n-k+j)!}.
\end{align*}
\end{satz}

Analogous results hold in the $B_n$-case.
Let $Y_1,\ldots,Y_n$ be random vectors in $\R^d$ that satisfy assumption~\ref{label_GP1} a.s.;
see Lemma~\ref{Lemma_AS_General_Position}  for a natural general setting in which~\ref{label_GP1} holds.
We define
the \textit{Weyl random cone} $\Dnb$ of type $B_n$  as follows: Among the cones of the random Weyl tessellation $\mathcal{W}^B(Y_1,\ldots,Y_n)$ choose one uniformly at random. 
The next theorems are the analogues to Theorems~\ref{theorem:exp_k-faces_Weyl_cones_A} and~\ref{Theorem_sizefunction_Dn_An-1} in the $B_n$-case.

\begin{satz}\label{theorem:exp_k-faces_Weyl_cones_B}
Let $Y_1,\ldots,Y_n$ be random vectors in $\R^d$ that satisfy \ref{label_GP1} a.s.
Then, for all $k=1,\ldots,d$, the expected number of $k$-faces of the Weyl random cone $\Dnb$ of type $B_n$ is given by
\begin{align}\label{Eq_Kor_Number_faces}
\E f_k(\Dnb)=\frac{2^{d-k}\tbinom{n}{d-k}D^B(n-d+k,k)}{D^B(n,d)}\frac{n!}{(n-d+k)!}.
\end{align}
\end{satz}

\begin{satz} \label{Theorem_sizefunction_Dn}
Let $Y_1,\ldots,Y_n$ be random vectors in $\R^d$ that satisfy~\ref{label_GP1} a.s.\ and
let $\Dnb$ be a Weyl random cone of type $B_n$ in $\R^d$. Then, for all $1\le j\le k\le d$, it holds that
\begin{align}\label{Eq_Theorem_Sizefunction_Dn}
\E Y_{d-k+j,\,d-k}(\Dnb)=\frac{2^{k-j}\tbinom{n}{k-j}D^B(n-k+j,j)}{2D^B(n,d)}\frac{n!}{(n-k+j)!}.
\end{align}
\end{satz}

From Theorems~\ref{Theorem_sizefunction_Dn_An-1} and~\ref{Theorem_sizefunction_Dn} we will derive the expected values of several interesting  geometrical functionals as special cases.

\begin{kor}\label{cor:quermass_weyl_cones}
The expected conic quermassintegrals of the Weyl random cones $\Dna$ and $\Dnb$ are given by
\begin{align*}
\E U_j(\Dna)=\frac{D^A(n,d-j)}{2D^A(n,d)},\quad\E U_j(\Dnb)=\frac{D^B(n,d-j)}{2D^B(n,d)},\quad j=0,\ldots,d-1.
\end{align*}
\end{kor}

\begin{proof}
Replacing $k$ and $j$ in (\ref{Eq_Theorem_Sizefunction_Dn}) both by $d-j$, we obtain
\begin{align*}
\E U_j(\Dna)=\E Y_{d,\,j}(\Dna)=\E Y_{d,\,d-(d-j)}(\Dna)=\frac{D^A(n,d-j)}{2D^A(n,d)},
\end{align*}
which proves the claim for $\Dna$. The $B_n$-case follows in the same way.
\end{proof}
As another consequence of Theorems~\ref{Theorem_sizefunction_Dn_An-1} and~\ref{Theorem_sizefunction_Dn}, we can compute the expectations of the so-called \textit{conic intrinsic volumes}.
Let $C\subset \R^d$ be a polyhedral cone.
Let $\Pi_C:\R^d\to C$ be the Euclidean projection on $C$,
that is $\Pi_C(x)$ is the unique vector in $C$ minimizing the Euclidean distance to $x\in\R^d$.
Then, the $j$-th conic intrinsic volume of $C$ is defined by
\begin{align*}
\upsilon_j(C):=\sum_{F\in\F_j(C)}\PP(\Pi_C(g)\in\relint F),\quad j=0,\ldots,d,
\end{align*}
where $g$ is a  $d$-dimensional standard Gaussian random vector and $\relint F$ denotes the interior of the face $F$
taken with respect to its linear hull.  For an extensive account on conic intrinsic volumes,
their properties and applications,
we refer to~\cite[Section~6.5]{Schneider2008} as well as~\cite{Amelunxen2017} and~\cite{Amelunxen2014}.

\begin{kor}\label{cor:intr_vol_weyl_random_cones}
The expected conic intrinsic volumes of the Weyl random cones $\Dna$ and $\Dnb$ are given by
\begin{align*}
\E\upsilon_0(\Dna)&=\frac{D^A(n,d)-D^A(n,d-1)}{2D^A(n,d)},\quad\E \upsilon_j(\Dna)=\frac{\stirling n{n-d+j}}{D^A(n,d)},\quad j=1,\ldots,d,\\
\E\upsilon_0(\Dnb)&=\frac{D^B(n,d)-D^B(n,d-1)}{2D^B(n,d)},\quad\E \upsilon_j(\Dnb)=\frac{\stirlingb n{n-d+j}}{D^B(n,d)},\quad j=1,\ldots,d.
\end{align*}
\end{kor}

\begin{proof} We use the conic Crofton formula~\cite[p.~261]{Schneider2008}, which states that
\begin{align*}
\upsilon_d(C)=U_{d-1}(C),\quad \upsilon_{d-1}(C)=U_{d-2}(C),\quad\upsilon_j(C)=U_{j-1}(C)-U_{j+1}(C),
\end{align*}
for all $j = 1,\ldots,d-2$ and any cone $C\subset\R^d$. Thus, for $j = 1,\ldots,d-2$, we obtain
\begin{align*}
\E \upsilon_j(\Dna)=\E U_{j-1}(\Dna)-\E U_{j+1}(\Dna)=\frac{D^A(n,d-j+1)-D^A(n,d-j-1)}{2D^A(n,d)}=\frac{\stirling n{n-d+j}}{D^A(n,d)}.
\end{align*}
For $j=d-1$ and $j=d$, the claim follows in a similar way. For the case $j=0$, we use the relation $\upsilon_0(C)+\upsilon_1(C)+\ldots+\upsilon_d(C)=1$ which holds for all cones $C\subset\R^d$.
The $B_n$-case follows analogously.
\end{proof}

The last special case of Theorems~\ref{Theorem_sizefunction_Dn_An-1} and~\ref{Theorem_sizefunction_Dn} we want to state
deals with the expected sums of solid angles of the faces of the random Weyl cones.
The \textit{solid angle} of a cone $C\subset\R^d$ is defined as
\begin{align*}
\alpha(C):=\P(Z\in C),
\end{align*}
where $Z$ is a uniformly distributed random vector on the unit sphere in the linear hull of $C$.
Following~\cite{HugSchneider2016}, the functionals $\Lambda_k(C)$ are defined by
\begin{align*}
\Lambda_k(C)=\sum_{F\in\cF_k(C)}\alpha(F).
\end{align*}

\begin{kor} Let $\Dna$ and $\Dnb$ be the Weyl random cones of types $A_{n-1}$ and $B_n$, respectively. Then, it holds that
\begin{align*}
\E\Lambda_k(\Dna) =\frac{n!\binom {n-1}{d-k}}{(n-d+k)!}\frac{1}{D^A(n,d)},\quad\E\Lambda_k(\Dnb)=\frac{2^{d-k}n!\binom n{d-k}}{(n-d+k)!}\frac{1}{D^B(n,d)},\quad k=1,\ldots,d.
\end{align*}
\end{kor}

\begin{proof}
Since $\alpha(F)=U_{k-1}(F)$ holds for any $k$-dimensional face $F$, we obtain that
\begin{align*}
\E\Lambda_k(\Dna)=\E\sum_{F\in\cF_k(\Dna)}\alpha(F)=\E\sum_{F\in\cF_k(\Dna)}U_{k-1}(F)=\E Y_{k,k-1}(\Dna).
\end{align*}
Then, the claim follows by replacing $k$ by $d-k+1$ and setting $j=1$ in Theorem~\ref{Theorem_sizefunction_Dn_An-1}. The $B_n$-case is similar.
\end{proof}

\subsection{Dual Weyl random cones}\label{sec:dual_weyl_random_cones}
We are now going to define two more families of random cones, denoted by $\mathcal{C}_n^A$ and
$\mathcal{C}_n^B$. As it turns out, these cones are dual to the Weyl random cones $\Dna$ and $\Dnb$ in distribution.
We start with the $A_{n-1}$-case. Let $Y_1,\ldots,Y_n$ be random vectors in $\R^d$ satisfying assumption~\ref{label_GP1_A} a.s. Additionally, assume that $Y_1,\ldots,Y_n$ are \textit{exchangeable}, that is,
\begin{align*}
(Y_1,\ldots,Y_n)\eqdistr (Y_{\sigma(1)},\ldots,Y_{\sigma(n)})
\end{align*}
for every permutation $\sigma$ of the set $\{1,\ldots,n\}$.
Let $\mathcal{G}_n^A$ be a random cone defined by
$$
\mathcal{G}_n^A := \{v\in\R^d:\langle v,Y_1\rangle\le\ldots\le\langle v,Y_n\rangle\}.
$$
Using the exchangeability we have
\begin{equation}\label{eq:probab_G_non_zero}
\P\big(\mathcal{G}_n^A \neq\{0\}\big)
=\frac{1}{n!}\sum_{\sigma\in \mathcal S_n}\P\big(\{v\in\R^d:\langle v,Y_{\sigma(1)}\rangle \le \ldots\le \langle v,Y_{\sigma(n)}\rangle\}\neq\{0\}\big)
=\frac{D^A(n,d)}{n!},
\end{equation}
where we applied Theorem~\ref{theorem:number_cones_A} in the last step. The dual cone of $\mathcal{G}_n^A$ is given by
$$
\big(\mathcal{G}_n^A\big)^\circ = \pos\{Y_1-Y_2,\ldots,Y_{n-1}-Y_n\}.
$$
This follows from the well-known fact that
\begin{equation}\label{eq:dual_of_pos_hull}
(\pos\{x_1,\ldots,x_n\})^\circ=\{v\in\R^d:\langle v,x_i\rangle\le 0 \text{ for all } i=1,\ldots,n\}.
\end{equation}
We can now state  an alternative description of the Weyl cone $\Dna$.
\begin{prop}\label{Prop:Equivalent_Def_Dna}
Let $Y_1,\ldots,Y_n$ be random vectors in $\R^d$ which are exchangeable and  satisfy \ref{label_GP1_A} a.s.
Then, $\Dna$ has the same distribution as $\mathcal{G}_n^A$ conditioned on the event that $\mathcal{G}_n^A\neq \{0\}$.
\end{prop}
\begin{proof}
The method of proof is similar to that of~\cite[Theorem~3.1]{HugSchneider2016}.
Let  $\P_Y$ be the joint probability law of $(Y_1,\ldots,Y_n)$ on $(\R^d)^n$.
For a Borel set $B$ of polyhedral cones that does not contain $\{0\}$, we have
\begin{align*}
&\PP(\mathcal{G}_n^A \in B|\mathcal{G}_n^A\neq \{0\})\\
&	\quad=\frac{1}{\P(\mathcal G_n^A\neq \{0\})}
\int_{(\R^d)^n}\1_B(\{v\in\R^d:\langle v,y_1\rangle\le\ldots\le\langle v,y_n\rangle\})\,\P_Y(\text{d}(y_1,\ldots,y_n))\\
&	\quad=\frac{n!}{D^A(n,d)}\int_{(\R^d)^n}\frac{1}{n!}\sum_{\sigma\in \mathcal S_n}\1_B(\{v\in\R^d:\langle v,y_{\sigma(1)}\rangle\le\ldots\le\langle v,y_{\sigma(n)}\rangle\})\,\P_Y(\text{d}(y_1,\ldots,y_n))\\
&	\quad=\int_{(\R^d)^n}\frac{1}{D^A(n,d)}\sum_{C\in \mathcal{W}^A(y_1,\ldots,y_n)}\1_B(C)\,\P_Y(\text{d}(y_1,\ldots,y_n))\\
&	\quad=\PP(\Dna\in B),
\end{align*}
where we used~\eqref{eq:probab_G_non_zero} and the exchangeability of $(Y_1,\ldots,Y_n)$ in the second step.
\end{proof}

We now define the \textit{dual Weyl cone} $\mathcal{C}_n^A$ of type $A_{n-1}$ as the random cone whose distribution is that of
$
\pos\{Y_1-Y_2,\ldots,Y_{n-1}-Y_n\}
$
conditioned on the event that this positive hull is not equal to $\R^d$.
The next proposition follows from the above discussion.

\begin{prop}\label{Prop_Dual_Of_Dna}
Let $Y_1,\ldots,Y_n$ be random vectors in $\R^d$ which are exchangeable and  satisfy \ref{label_GP1_A} a.s.
Then, $\mathcal{C}_n^A$ has the same distribution as $\big(\Dna\big)^\circ$.
\end{prop}

In the $B_n$-case let $Y_1,\ldots,Y_n$ be random vectors in $\R^d$ satisfying assumption~\ref{label_GP1} a.s. Additionally, assume that $Y_1,\ldots,Y_n$ are \textit{symmetrically exchangeable}, that is,
\begin{align*}
(Y_1,\ldots,Y_n)\eqdistr (\eps_1Y_{\sigma(1)},\ldots,\eps_nY_{\sigma(n)})
\end{align*}
for all permutations $\sigma$ of the set $\{1,\ldots,n\}$ and all vectors of signs $\eps=(\eps_1,\ldots,\eps_n)\in\{\pm 1\}^n$.
Let $\mathcal{G}_n^B$ be a random cone defined by
$$
\mathcal{G}_n^B := \{v\in\R^d:\langle v,Y_1\rangle\le\ldots\le\langle v,Y_n\rangle\le 0\}.
$$
Using the symmetric exchangeability we have
\begin{equation*}
\P\big(\mathcal{G}_n^B \neq\{0\}\big)
=\frac{1}{2^n n!}\sum_{\eps,\sigma}\P\big(\{v\in\R^d:\langle v,\eps_1Y_{\sigma(1)}\rangle \le \ldots\le \langle v,\eps_nY_{\sigma(n)}\rangle\le 0\}\neq\{0\}\big)
=\frac{D^B(n,d)}{2^n n!},
\end{equation*}
where we applied Theorem~\ref{Theorem_Number_Weyl_Cones} in the last step. It follows from~\eqref{eq:dual_of_pos_hull} that
the dual cone of $\mathcal{G}_n^B$ is given by
$$
\big(\mathcal{G}_n^B\big)^\circ = \pos\{Y_1-Y_2,\ldots,Y_{n-1}-Y_n,Y_n\}.
$$
Define the \textit{dual Weyl cone} $\mathcal C_n^B$ of type $B_n$ as the random cone whose distribution is that of
$
\pos\{Y_1-Y_2,\ldots,Y_{n-1}-Y_n,Y_n\}
$
conditioned on the event that this positive hull is not equal to $\R^d$. The following is the analogue
to Propositions~\ref{Prop:Equivalent_Def_Dna} and~\ref{Prop_Dual_Of_Dna} and can be proven similarly.

\begin{prop}\label{Prop_Dual_Of_Dnb}
Let $Y_1,\ldots,Y_n$ be random vectors in $\R^d$ that are symmetrically
exchangeable and satisfy assumption \ref{label_GP1} a.s.
Then, $\Dnb$ has the same distribution as $\mathcal{G}_n^B$ conditioned on the event $\mathcal{G}_n^B\neq \{0\}$,
while $\mathcal{C}_n^B$ has the same distribution as $(\mathcal{D}_n^B)^\circ$.
\end{prop}

Using well-known duality relations for the number of faces,
the conic intrinsic volumes and the conic quermassintegrals,
we can compute the expected values of geometric functionals of the random cones $\mathcal{C}_n^A$ and $\mathcal{C}_n^B$.
We assume that $Y_1,\ldots,Y_n$ are  exchangeable and satisfy~\ref{label_GP1_A} a.s.~(in the $A_{n-1}$-case) or that $Y_1,\ldots,Y_n$ are symmetrically exchangeable and satisfy~\ref{label_GP1} a.s.~(in the $B_n$-case).

\begin{kor}
The expected number of $k$-faces of $\mathcal{C}_n^A$ and $\mathcal C_n^B$, for $k=0,\ldots,d-1$, are given by
\begin{align*}
\E f_k(\mathcal{C}_n^A)=\frac{\binom{n-1}{k}D^A(n-k,d-k)}{D^A(n,d)}\frac{n!}{(n-k)!},\quad\E f_k(\mathcal{C}_n^B)=\frac{2^k\binom{n}{k}D^B(n-k,d-k)}{D^B(n,d)}\frac{n!}{(n-k)!}.
\end{align*}
\end{kor}

\begin{proof}
The formulas follow from Theorems~\ref{theorem:exp_k-faces_Weyl_cones_A} and~\ref{theorem:exp_k-faces_Weyl_cones_B} together with Propositions~\ref{Prop_Dual_Of_Dna} and~\ref{Prop_Dual_Of_Dnb}, using the well-known duality relation $f_k(C)=f_{d-k}(C^\circ)$ for any cone $C\subset \R^d$.
\end{proof}

\begin{kor}
The expected conic quermassintegrals of $\mathcal{C}_n^A$ and $\mathcal C_n^B$ are given by
\begin{align*}
\E U_j(\mathcal{C}_n^A)=\frac{D^A(n,d)-D^A(n,j)}{2D^A(n,d)},\quad\E U_j(\mathcal{C}_n^B)=\frac{D^B(n,d)-D^B(n,j)}{2D^B(n,d)}, \quad j=1,\ldots,d.
\end{align*}
\end{kor}

\begin{proof}
This follows directly from the duality relation $U_j(C)+U_{d-j}(C^\circ)=\frac 12$ which holds for any cone $C\subset\R^d$ which is not a linear subspace, see~\cite[Eq.~(5)]{HugSchneider2016}. Note that since the random Weyl cones $\Dna$ and $\Dnb$ are a.s.~$d$-dimensional (but not equal to $\R^d$), their dual cones $\mathcal C_n^A$ and $\mathcal C_n^B$, respectively, are a.s.~pointed. This means that $\{0\}$ is the only linear subspace they contain, in particular, $\mathcal C_n^A$ and $\mathcal C_n^B$ are a.s.\ no linear subspaces.
\end{proof}

\begin{kor}
The expected conic intrinsic volumes of $\mathcal{C}_n^A$ and $\mathcal C_n^B$ are given by
\begin{align*}
\E\upsilon_d(\mathcal{C}_n^A)&=\frac{D^A(n,d)-D^A(n,d-1)}{2D^A(n,d)},\quad\E \upsilon_j(\mathcal{C}_n^A)=	\stirling n{n-j}\frac{1}{D^A(n,d)}, \quad j=0,\ldots,d-1,\\
\E\upsilon_d(\mathcal{C}_n^B)&=\frac{D^B(n,d)-D^B(n,d-1)}{2D^B(n,d)},\quad\E \upsilon_j(\mathcal{C}_n^B)=	\frac{\stirlingb n{n-j}}{D^B(n,d)}, \quad j=0,\ldots,d-1.
\end{align*}
\end{kor}

\begin{proof}
This follows directly from the well-known duality relation $\upsilon_j(C)=\upsilon_{d-j}(C^\circ)$ for any cone $C\subset\R^d$, see, for example,~\cite[Eq.~(2.9)]{Amelunxen2017}.
\end{proof}

\subsection{Remarks}\label{sec:remarks}
The papers~\cite{KVZ15} and~\cite{KVZ17} studied convex hulls of the $d$-dimensional random walks (and bridges) of the form $Y_1, Y_1+Y_2,\ldots, Y_1+\ldots+Y_n$, where $Y_1,\ldots,Y_n$ are random vectors satisfying certain exchangeability and general position conditions; see also~\cite{vysotsky_zaporozhets}. The main results of these works are formulas for the probability that such a convex hull contains the origin, as well as for the expected number of $j$-faces of the convex hull. These formulas (which are distribution-free) also involve the numbers $D^A (n,d)$ and $D^B(n,d)$.
Above, we described the dual cones of $\Dna$, respectively $\Dnb$. Under natural exchangeability assumptions on $Y_1,\ldots,Y_n$, these turned out to be the positive hulls of $Y_1-Y_2, \ldots, Y_{n-1}-Y_n, Y_n$, respectively $Y_{1}-Y_2,\ldots,Y_{n-1}-Y_n$, (conditioned on the event that this positive hull is not $\R^d$). For these positive hulls, we stated formulas for the expected values of some geometric functionals, thus showing that the differences of exchangeable random vectors also exhibit a distribution-free behavior.

Let us finally mention that it is possible to extend the results of the present paper to Weyl tessellations corresponding to the reflection groups of the product type $B_{n_1}\times \ldots \times B_{n_r} \times A_{k_1-1}\times\ldots \times A_{k_l-1}$. The conical tessellations of product type are just unions of the tessellations corresponding to the individual  factors.
In particular, Weyl tessellations of type $B_1^n$ coincide with the tessellations studied by Cover and Efron~\cite{CoverEfron_paper} and  Hug and Schneider~\cite{HugSchneider2016}. Thus, their results become special cases of this more general setting. We refrain from stating the results in the product type setting since they require introducing heavy notation. The above observation explains the similarity
in the formulas for the expected geometrical functionals of Weyl random cones to the respective results on the Schl\"afli and Cover-Efron random cones which Hug and Schneider stated and proved in~\cite[Section 4]{HugSchneider2016}.

\subsection{Outline of the paper}
The rest of the paper is mostly devoted to the proofs of the results stated in Section~\ref{sec:main_results}.
We will not prove the results in the order in which they are stated above.
In Section~\ref{sec:proof_faces_multiplicity} we state and prove a formula which will be the key in proving most of the other results. In Section~\ref{sec:proof_weyl_tess} we will prove the results on the number of faces (Theorems~\ref{Theorem_Number_WeylFaces_Tessellation_An-1} and~\ref{Theorem_Number_WeylFaces_Tessellation_Bn}) in a Weyl tessellation of the respective type. Section~\ref{sec:proofs_size_funct} contains the proofs of the formulas for the expected size functionals of Weyl random cones (Theorems~\ref{Theorem_sizefunction_Dn_An-1} and~\ref{Theorem_sizefunction_Dn}), while Section~\ref{Subsection_Proof_GeneralPosition} is dedicated to proving the results concerning general position.




\section{Faces of the Weyl tessellations counted with multiplicity}\label{sec:proof_faces_multiplicity}

Most results from Section~\ref{sec:main_results} heavily rely on (or directly follow from) the following propositions. They give a formula for the number of $k$-faces in the Weyl tessellations, where each face is taken with the multiplicity equal to the number of Weyl cones containing it.

\begin{prop}\label{prop:faces_with_multiplicity_A}
Let $y_1,\ldots,y_n\in\R^d$ satisfy assumption~\ref{label_GP1_A}. Then, for all $k\in \{1,\ldots,d\}$, it holds that
\begin{align*}
&\sum_{F\in\F^A_k(y_1,\ldots,y_n)}\:\sum_{C\in\mathcal W^A(y_1,\ldots,y_n)}\1_{\{F\subseteq C\}}=\binom{n-1}{d-k}\frac{n!}{(n-d+k)!}D^A(n-d+k,k).
\end{align*}
\end{prop}

\begin{prop}\label{prop:faces_with_multiplicity_B}
Let $y_1,\ldots,y_n\in\R^d$ satisfy assumption~\ref{label_GP1}.  Then, for all $k\in\{1,\ldots,d\}$, it holds that
\begin{align*}
&\sum_{F\in\F^B_k(y_1,\ldots,y_n)}\:\sum_{C\in\mathcal W^B(y_1,\ldots,y_n)}\1_{\{F\subseteq C\}}=2^{d-k}\binom{n}{d-k}\frac{n!}{(n-d+k)!}D^B(n-d+k,k).
\end{align*}
\end{prop}

Firstly, we note that the formulas for the number of cones in the Weyl tessellation as well as the formulas for the expected number of $k$-faces of the Weyl random cones $\Dna$ and $\Dnb$ are direct consequences of Propositions~\ref{prop:faces_with_multiplicity_A} and~\ref{prop:faces_with_multiplicity_B}.

\begin{proof}[Proof of Theorems~\ref{theorem:number_cones_A} and~\ref{Theorem_Number_Weyl_Cones} assuming Propositions~\ref{prop:faces_with_multiplicity_A} and~\ref{prop:faces_with_multiplicity_B}]
Both theorems follow directly from the respective Propositions~\ref{prop:faces_with_multiplicity_A}
and~\ref{prop:faces_with_multiplicity_B} with $k=d$.
\end{proof}

\begin{proof}[Proof of Theorems~\ref{theorem:exp_k-faces_Weyl_cones_A} and~\ref{theorem:exp_k-faces_Weyl_cones_B} assuming Propositions~\ref{prop:faces_with_multiplicity_A} and~\ref{prop:faces_with_multiplicity_B}]
 Both follow directly from the respective Propositions~\ref{prop:faces_with_multiplicity_A} and~\ref{prop:faces_with_multiplicity_B} using that the total number of cones in the Weyl tessellations is given by $D^A(n,d)$ and $D^B(n,d)$ a.s.\ (by Theorems~\ref{theorem:number_cones_A} and~\ref{Theorem_Number_Weyl_Cones}), respectively, and additionally, that $\Dna$ and $\Dnb$ are
chosen uniformly among all cones of the respective Weyl tessellation.
\end{proof}

The rest of this section is dedicated to proving Propositions~\ref{prop:faces_with_multiplicity_A}
and~\ref{prop:faces_with_multiplicity_B}.
Most of the time we will deal with the $B_n$-case since the $A_{n-1}$-case follows in a similar but somewhat simpler way.
In Section~\ref{sec:preliminary_comments}, we introduce the necessary notation and terminology.
Section~\ref{sec:char_Weyl_faces} gives a characterization of the $k$-faces of the Weyl tessellations.
In Section~\ref{sec:proof_faces_multiplicity_reduction}  we reduce Propositions~\ref{prop:faces_with_multiplicity_A} and~\ref{prop:faces_with_multiplicity_B} to a problem of counting the number of $k$-faces of Weyl chambers
intersected by a linear subspace.
Finally, in Section~\ref{sec:proof_faces_multiplicity_final_part} we complete the proof of
Propositions~\ref{prop:faces_with_multiplicity_A} and~\ref{prop:faces_with_multiplicity_B}.

\subsection{Preliminary comments and notation}\label{sec:preliminary_comments}

We collect some facts and definitions concerning polyhedral cones, conical tessellations  and general position. Some of these definitions and results were already used in Section~\ref{sec:main_results}, but we want to give formal definitions since they
are necessary to understand the subsequent proofs.

\subsubsection*{Polyhedral cones}
A \textit{polyhedral cone} (or, for simplicity, just a \textit{cone}) $C\subseteq\R^d$ is a finite intersection of closed half-spaces whose boundaries pass through the origin.
The dimension of a cone $C$ is defined as the dimension of its linear hull, i.e.~$\dim C=\dim\lin(C)$.
The \textit{faces} of $C$ are obtained by replacing some of the half-spaces, whose intersection defines the polyhedral cone, by their boundaries and taking the intersection. We denote by $\F_k(C)$ the set of all $k$-dimensional faces of $C$.
Recall that $f_k(C)=\#\F_k(C)$ denotes the number of $k$-faces of $C$.
The
\textit{dual cone} of a cone $C\subseteq \R^d$,
defined by $C^\circ=\{x\in\R^d:\langle x,y\rangle\le 0\;\forall y\in C\}$.

Furthermore, let $\linsp(C)=C\cap (-C)$ denote the \textit{lineality space} of $C$, which is the linear subspace contained in  $C$ and having the  maximal possible dimension. Additionally, $\linsp(C)$ is contained in every face of $C$. A cone $C$ is \textit{pointed} if it does not contain a non-trivial linear subspace, i.e. if $\{0\}$ is a 0-dimensional face, or equivalently, if $\linsp(C)=\{0\}$.

For the hyperplanes and half-spaces induced by a vector $x\in\R^d\backslash\{0\}$ we use the notation
\begin{align*}
x^\perp=\{y\in\R^d:\langle y,x\rangle =0\},\quad x^-=\{y\in\R^d:\langle y,x\rangle\le 0\}.
\end{align*}

%
%




\subsubsection*{General position}
Throughout this paper we will use the notion of \textit{general position} in several different contexts. These notions are well-known in convex geometry, but to prevent confusion, we will provide a list which is complete in the sense of this manuscript.
\begin{enumerate}[ leftmargin=40pt]
\item[(i)] Vectors $x_1,\ldots,x_m\in\R^d$ are in \textit{general position} if for any $k\le d$ and $1\le i_1<\ldots<i_k\le m$ the set of vectors $x_{i_1},\ldots,x_{i_k}$ is linearly independent.
\item[(ii)] Linear hyperplanes $H_1,\ldots,H_m\subset \R^d$ are in \textit{general position} if
\begin{align*}
\dim(H_{i_1}\cap\ldots\cap H_{i_k})=d-k
\end{align*}
holds for any $k\le d$ and $1\le i_1<\ldots<i_k\le m$. Equivalently, the unit normal vectors of $H_1,\ldots,H_m$ are in general position.
\item[(iii)] Two linear subspaces $L$ and $L'$ in $\R^d$ are in \textit{general position} if
\begin{align*}
\dim (L\cap L')=\max\{0,\dim L + \dim L'-d\}.
\end{align*}
\item[(iv)] A linear subspace $L\subset \R^d$ is in \textit{general position} with respect to a hyperplane arrangement $\cA$ in $\R^d$ if $L$ is in general position with respect to $\bigcap_{H\in\cB}H$, for all finite subsets $\mathcal B\subset \cA$.
\end{enumerate}

\begin{bsp}\label{Remark_General_position_Uniform_Subspace}
Let $L$ be a random $k$-dimensional linear subspace with distribution $\nu_k$,
which is the unique rotation invariant Borel measure
on the Grassmannian $G(d,k)$ of all $k$-dimensional linear subspaces of $\R^d$.
Then, $L$ is in general position to any fixed linear subspace
and therefore also in general position to any fixed hyperplane arrangement, with probability $1$.  This follows directly from~\cite[Lemma 13.2.1]{Schneider2008}.
\end{bsp}

\begin{bem}
With the introduced terminology, assumptions~\ref{label_GP1_A}  and~\ref{label_GP1} can be rewritten
as follows:
\begin{enumerate}[ leftmargin=50pt]
\item[(A1)] For every  $\sigma\in\mathcal{S}_n$ the vectors $y_{\sigma(1)}-y_{\sigma(2)},y_{\sigma(2)}-y_{\sigma(3)},\ldots,y_{\sigma(n-1)}-y_{\sigma(n)}$ are in general position, and $n\ge d+1$.
\item[(B1)] For every $\varepsilon=(\varepsilon_1,\ldots,\varepsilon_n)\in\{\pm 1\}^n$ and $\sigma\in\mathcal{S}_n$ the vectors $\varepsilon_1y_{\sigma(1)}-\varepsilon_2y_{\sigma(2)},\varepsilon_2y_{\sigma(2)}-\varepsilon_3y_{\sigma(3)},\ldots,\varepsilon_{n-1}y_{\sigma(n-1)}-\varepsilon_ny_{\sigma(n)}, \varepsilon_ny_{\sigma(n)}$ are in general position, and $n\ge d$.
\end{enumerate}
\end{bem}

\subsubsection*{Conical tessellations}
We already saw in the introduction that a hyperplane arrangement $\A$ generates  a finite  set of cones. Denote by $\mathcal{R}(\mathcal{A})$ the set of open connected components ( ``regions" or ``chambers") of the complement $\R^d\setminus\bigcup_{H\in\mathcal{A}}H$ of the hyperplanes. Then, $\xbar{\mathcal{R}}(\mathcal{A}):=\{\bar{R}:R\in\mathcal{R}(\mathcal{A})\}$ is the \textit{conical tessellation} generated by $\mathcal{A}$, where $\bar{R}$ denotes the closure of $R$. The set of faces $\F(\xbar{\mathcal{R}}(\A))$ of $\xbar{\mathcal{R}}(\mathcal{A})$ is defined as the union of the sets of faces of all polyhedral cones $C\in\xbar{\mathcal{R}}(\mathcal{A})$.
Thus, the conical tessellation $\mathcal{\xbar{\mathcal{R}}}(\mathcal{A})$ generated by  a hyperplane arrangement $\A=\{H_1,\dots,H_n\}$ in $\R^d$ consists precisely of the cones of the form
\begin{align*}
\bigcap_{i=1}^n\eps_iH_i^-,\quad (\eps_1,\dots,\eps_n)\in\{\pm 1\}^n,
\end{align*}
which have non-empty interiors. Here, $H_i^-$ is the notation for one of the closed half-spaces bounded by $H_i$.

\begin{bem}\label{remark:cones_full-dim}
If we additionally assume $H_1,\dots,H_n$ to be in general position, then all cones different from $\{0\}$ of the form $\bigcap_{i=1}^n\eps_iH_i^-$ are already $d$-dimensional, see~\cite[(2.14)]{HugSchneider2016}, and thus, are the cones of the conical tessellation $\mathcal{\xbar{\mathcal{R}}}(\mathcal{A})$.
\end{bem}

For a linear subspace $L\subset\R^d$ in general position with respect to a hyperplane arrangement $\cA$ in $\R^d$, the \textit{induced hyperplane arrangement} in $L$ is defined by $\mathcal{A}|L:=\{H\cap L: H\in\mathcal{A}\}$. The general position assumption guarantees that the subspaces $H\cap L$ are pairwise distinct and indeed hyperplanes in $L$. It is easy to check that the cones of the conical tessellation in $L$ generated by $\cA|L$ are obtained by intersecting the cones of $\xbar{\mathcal{R}}(\A)$ with $L$:
$$
\xbar{\mathcal{R}}(\A|L) = \{C\cap L:C\in\xbar{\mathcal{R}}(\A),C\cap L\neq\{0\}\}.
$$
If $\cA=\{H_1,\ldots,H_n\}$ and $y_1,\ldots,y_n\in\R^d$ denote respective normal vectors of $H_1,\ldots,H_n$,  then their orthogonal projections on $L$, denoted by $\Pi_L(y_1),\ldots,\Pi_L(y_n)$, are normal vectors of $H_1\cap L,\ldots,H_n\cap L$ in $L$, respectively.

\subsection{Characterizing the faces of the Weyl tessellation}\label{sec:char_Weyl_faces}
In this section we will provide an explicit characterization of the $k$-faces of the Weyl tessellations
$\mathcal{W}^A(y_1,\ldots,y_n)$  and $\mathcal{W}^B(y_1,\ldots,y_n)$. Also, we will determine the number of cones in the tessellation that contain a given face.

We will concentrate on  the $B_n$-case. Let $y_1,\ldots,y_n\in\R^d$ satisfy~\ref{label_GP1}. As mentioned in the introduction, the cones in the Weyl tessellation $\mathcal{W}^B(y_1,\ldots,y_n)$ of type $B_n$
are the cones different from $\{0\}$ of the form
\begin{align*}
D^B_{\varepsilon,\,\sigma}:=\{v\in\R^d:\langle v,\varepsilon_1y_{\sigma(1)}\rangle\le \ldots\le \langle v,\varepsilon_ny_{\sigma(n)}\rangle\le 0\},\quad \varepsilon=(\varepsilon_1,\ldots,\varepsilon_n)\in\{\pm 1\}^n,\sigma\in\mathcal{S}_n.
\end{align*}
This follows directly from Remark~\ref{remark:cones_full-dim} and the definition of the generating hyperplane arrangement $\cA^B(y_1,\dots,y_n)$ in~\eqref{eq:Weyl_arrangement_B}. We will also refer to these cones as \textit{Weyl cones of type $B_n$} or just \textit{Weyl cones} when the type we are referring to is obvious from the context. Define the linear functionals $f_i:\R^d\to \R$ by
$$
f_i=f_i(v):=\langle v,y_i\rangle, \qquad  i=1,\ldots,n.
$$
We now introduce notation for cones which, as we will show in a moment, represent the $k$-faces of the Weyl tessellation.
For a collection of indices $1\le l_1<\ldots<l_ {n-d+k}\le n$, a vector of signs $\varepsilon\in\{\pm 1\}^n$ and a permutation $\sigma\in\mathcal{S}_n$, we define
\begin{align}\label{Eq_Rep_Faces_Bn}
F^B_{\eps,\,\sigma}(l_1,\ldots,l_{n-d+k}):=\{v\in\R^d&:\varepsilon_1f_{\sigma(1)}=\ldots=\varepsilon_{l_1}f_{\sigma(l_1)}\le \varepsilon_{l_1+1}f_{\sigma(l_1+1)}=\ldots =\varepsilon_{l_2}f_{\sigma(l_2)}\notag\\
&	\;\;\le\ldots\le \varepsilon_{l_{n-d+k-1}+1}f_{\sigma(l_{n-d+k-1}+1)}= \ldots=\varepsilon_{l_{n-d+k}}f_{\sigma(l_{n-d+k})}\\
&	\;\;\le f_{\sigma(l_{n-d+k}+1)}=\ldots=f_{\sigma(n)}=0\}.\notag
\end{align}
If $l_{n-d+k}=n$, no $f_i$'s are required to be 0 (but, as always, all $\eps_i f_{\sigma(i)}$ are required to be non-negative).  Recall that $\F^B_k(y_1,\ldots,y_n)$ denotes the set of $k$-faces of the Weyl tessellation $\cW^B(y_1,\ldots,y_n)$.

\begin{prop}\label{Prop_Char_WeylFaces}
Let $1\le k\le d$ and let $y_1,\ldots,y_n\in\R^d$ satisfy the general position assumption \ref{label_GP1}.
Then, the following claims hold:
\begin{enumerate}[label=(\roman*)]
\item For every $F\in\F^B_k(y_1,\ldots,y_n)$ there exists a collection of indices $1\le l_1<\ldots<l_ {n-d+k}\le n$, a vector of signs $\varepsilon\in\{\pm 1\}^n$ and a permutation $\sigma\in\mathcal{S}_n$, such that $F=F^B_{\eps,\,\sigma}(l_1,\ldots,l_{n-d+k})$.
\item Let $1\le l_1<\ldots<l_{n-d+k}\le n$ and $\eps\in\{\pm 1\}^n$, $\sigma\in\mathcal{S}_n$. If $F^B_{\eps,\,\sigma}(l_1,\ldots,l_{n-d+k})\neq \{0\}$, then $F^B_{\eps,\,\sigma}(l_1,\ldots,l_{n-d+k})\in \F^B_k(y_1,\ldots,y_n)$.
\end{enumerate}
\end{prop}

\begin{proof}
We start by proving (i). Let $F\in\F^B_k(y_1,\ldots,y_n)$. Then there exists a Weyl cone $D\in\mathcal W^B(y_1,\ldots,y_n)$, such that $D=D^B_{\eps,\sigma}$ and $F\in\F^B_k(D)$ for suitable $\eps\in\{\pm 1\}^n$ and $\sigma\in\mathcal{S}_n$. Every face of a polyhedral cone is obtained by replacing some of the half-spaces whose intersection defines the cone by their boundaries, or in this case, equivalently, replacing some of the inequalities in the defining condition of $D^B_{\eps,\sigma}$ by equalities. Thus, we can find a number $1\le m\le n$ and a collection of indices $1\le l_1<\ldots<l_m\le n$, such that $F=F^B_{\eps,\sigma}(l_1,\ldots,l_m)$. It is left to show that $m=n-d+k$. We have
\begin{align*}
F\subseteq L_m:=\{v\in\R^d&:\eps_1f_{\sigma(1)}=\ldots=\eps_{l_1}f_{\sigma(l_1)}, \ldots,\eps_{l_{m-1}+1}f_{\sigma(l_{m-1}+1)}=\ldots=\eps_{l_m}f_{\sigma(l_m)},\\
&\;\;\;f_{\sigma(l_m+1)}=\ldots=f_{\sigma(n)}=0\}.
\end{align*}
Note that $L_m\neq \{0\}$ is the intersection of $n-m$ hyperplanes in general position due to~\ref{label_GP1}, which yields $\dim L_m=d-n+m$. Now, we want to show that the dimensions of $F$ and $L_m$ are equal. Due to \ref{label_GP1}, $L_m$ is in general position to the arrangement
\begin{align*}
\mathcal{A}_m:=\big\{\big(\eps_{l_1}y_{\sigma(l_1)}-\eps_{l_1+1}y_{\sigma(l_1+1)}\big)^\perp,\ldots,\big(\eps_{l_m}y_{\sigma(l_m)}-\eps_{l_m+1}y_{\sigma(l_m+1)}\big)^\perp\big\},
\end{align*}
where $\eps_{l_m+1}=0$ for $l_m=n$, and additionally, the hyperplanes of $\A_m$ are itself in general position. Thus, the hyperplanes in the induced arrangement $\A_m|L_m=\{H\cap L_m:H\in\A_m\}$ are in general position in $L_m$, which follows directly from the definition. Since $F$ is obviously different from $\{0\}$ and of the form $\bigcap_{H\in\cA_m}\delta_HH^-\cap L_m$, for suitable signs $\delta_H\in\{\pm 1\}$ and half-spaces $H^-$, Remark~\ref{remark:cones_full-dim} implies that $F$ is full-dimensional in $L_m$ (in this case meaning that  $\dim F = d-n+m$) and already a cone of the induced conical tessellation $\xbar{R}(\A_m|L_m)$. In particular, we have $k=\dim F=d-n+m$, which completes the proof of part (i).


The proof of (ii) is similar. Obviously $F^B_{\eps,\sigma}(l_1,\ldots,l_{n-d+k})$ is a face of the Weyl tessellation $\mathcal{W}^B(y_1,\ldots,y_n)$ for all $1\le l_1<\ldots<l_{n-d+k}\le n$ and $\eps\in\{\pm 1\}^n$, $\sigma\in\mathcal{S}_n$. If not $\{0\}$,  it is already a $k$-dimensional face, due to the general position arguments presented above.
\end{proof}

Next we want to evaluate the number of Weyl cones $C\in\mathcal W^B(y_1,\ldots,y_n)$ that contain a fixed $k$-face $F=F^B_{\eps,\sigma}(l_1,\ldots,l_{n-d+k})\in\F_k^B(y_1,\ldots,y_n)$, which will be necessary for the proof of Proposition~\ref{prop:faces_with_multiplicity_B}. In order to avoid heavy notation, we consider an example from which the general case should become evident.

\begin{bsp}\label{Example_WeylCones_containing_face}
Consider the case $n=7$, $d=6$, $k=2$ and the face of the Weyl tessellation $\mathcal W^B(y_1,\ldots,y_7)$ in $\R^6$ given by
\begin{align*}
F=\{v\in\R^6:\underbrace{-f_5=f_2}_{\text{group 1}}\le \underbrace{-f_3}_{\text{group 2}}\le \underbrace{f_4=f_1}_{\text{group 3}}\le \underbrace{f_6=f_7=0}_{\text{group 4}}\}.
\end{align*}
Assume that $F\neq \{0\}$.
Under  assumption \ref{label_GP1}, the cone $F$ is a $2$-dimensional face of the Weyl tessellation of type $B_7$, due to Proposition~\ref{Prop_Char_WeylFaces}. In particular, it is a $2$-face of the Weyl cone
\begin{align}\label{Eq_Bsp_Cone}
\{v\in\R^6:-f_5\le f_2\le -f_3\le f_4\le f_1\le f_6\le f_7\le 0\}.
\end{align}
However, it is also a $2$-face of
\begin{align*}
\{v\in\R^6:f_2\le -f_5\le -f_3\le f_4\le f_1\le -f_7\le f_6\le 0\}
\end{align*}and, more generally of any cone obtained from (\ref{Eq_Bsp_Cone}) by permuting the $f$'s inside the groups $(-f_5,f_2),(-f_3),(f_4,f_1),(f_6,f_7)$, and by changing any number of signs in the \textit{last} group. The total number of cones obtained in this way is $2!1!2!2!2^2$.

The question arises, if there are any other Weyl cones that contain $F$. We will show that the answer is ``no''. Indeed, if we change the sign of any $f_i$ that is not in the last group (for example, $f_1$), we see that $F$ is not contained in this cone any longer, e.g.
\begin{align*}
F\nsubseteq\{v\in\R^6:-f_5\le f_2\le -f_3\le f_4\le -f_1\le f_6\le f_7\le 0\}.
\end{align*}
Otherwise, $F$ would be contained in the hyperplanes $\{f_1\le 0\}$ and $\{-f_1\le 0\}$, and thus, $F\subseteq \{f_1 = 0\}$. This implies
\begin{align*}
F = \{v\in\R^d:-f_5=f_2\le -f_3\le f_4=f_1=f_6= f_7=0\}.
\end{align*}
The number of groups in this representation is strictly smaller than in the original one because the former group of $f_1$ joined  the last group. In fact, the cone on the right-hand side is a $1$-face of the Weyl tessellation, due to Proposition~\ref{Prop_Char_WeylFaces} under the general position assumption \ref{label_GP1}. This means that any cone obtained by altering a sign of any $f_i$ outside the last group does not contain $F$.

From now on, we can consider only the Weyl cones in whose representations the signs of all $f_i$'s are the same as in the original representation of $F$, except for the $f_i$'s in the last group.  Take such a cone and assume that in its representation we have an inequality $\pm f_j\leq \pm f_i$, while in the representation of $F$ we have the converse inequality $\pm f_i\leq \pm f_j$.  For example take the cone
\begin{align*}
\{v\in\R^d:-f_5=f_2\le -f_3\le f_4=f_6\le f_1=f_7=0\}
\end{align*}
which satisfies $f_6 \leq f_1$, while in the representation of $F$ we have $f_1\leq f_6$. We claim that $F$ is not contained in this cone.
Indeed, otherwise, $F$ would be contained in the hyperplane $\{f_1=f_6\}$, which implies that $$
F=\{v\in\R^6: -f_5=f_2 \le -f_3 \le f_4=f_1 = f_6=f_7=0\}.
$$
Again, the number of groups in this representation is strictly smaller than in the original representation of $F$. In fact, the cone on the right-hand side  is a $1$-face, similarly to the previous case. This is a contradiction, since $F$ is $2$-dimensional. This means that there are no other Weyl cones containing $F$. Generalizing this argument,  yields the following proposition.
\end{bsp}

\begin{prop}\label{Prop_WeylKegel_die_Seiten_enthalten_B_n}
Let $1\le k\le d$ and let $y_1,\ldots,y_n\in\R^d$ satisfy the general position assumption \ref{label_GP1}. Then, each $k$-face $F^B_{\eps,\,\sigma}(l_1,\ldots,l_{n-d+k})\in\mathcal{F}^B_k(y_1,\ldots,y_n)$ belongs to exactly
\begin{align*}
l_1!(l_2-l_1)!\cdot\ldots\cdot(n-l_{n-d+k})!2^{n-l_{n-d+k}}
\end{align*}
cones $C\in\mathcal W^B(y_1,\ldots,y_n)$.
\end{prop}

The same argument proves also the following proposition stating that there is a one-to-one correspondence between the $k$-faces of the Weyl tessellation and those combinatorial representations that lead to a non-trivial face.
\begin{prop}\label{Prop uniq}
Let $y_1,\ldots,y_n\in\R^d$ satisfy the general position assumption \ref{label_GP1}.
Let furthermore $F=F^B_{\eps,\sigma}(l_1,\ldots,l_{n-d+k})$ and $G = F^B_{\delta,\pi}(i_1,\ldots,i_{n-d+m})$ be such that  $F\neq \{0\}$, with some $\varepsilon,\delta\in\{\pm 1\}^n$, $\sigma,\pi\in\mathcal{S}_n$ and  $1\leq k \leq d$, $1\leq m\leq d$ and $1\le l_1<\ldots<l_ {n-d+k}\le n$, $1\le i_1<\ldots<i_ {n-d+m}\le n$.  If $F= G$, then $\sigma = \pi$, $k=m$,  $l_j = i_j$ for all admissible $j$, and
$\eps_{j}= \delta_{j}$ for all $1\leq j \leq l_{n-d+k}$.
\end{prop}

The $A_{n-1}$-case can be treated similarly and we only briefly describe the results.
For $y_1,\ldots,y_n\in\R^d$ satisfying~\ref{label_GP1_A},
a collection of indices $1\le l_1<\ldots<l_{n-d+k-1}\le n-1$ and $\sigma\in\mathcal{S}_n$, we define
\begin{align}\label{Eq_Rep_Weyl_Faces_An-1}
&F^A_{\sigma}(l_1,\ldots,l_{n-d+k-1})\\
&	\quad:=\{v\in\R^d:f_{\sigma(1)}=\ldots=f_{\sigma(l_1)}\le f_{\sigma(l_1+1)}=\ldots =f_{\sigma(l_2)}\le\ldots\le f_{\sigma(l_{n-d+k-1}+1)}=\ldots=f_{\sigma(n)}\},\notag
\end{align}
where we recall that the functionals $f_i$ are defined by $f_i=f_i(v):=\langle v,y_i\rangle$, $i=1,\ldots,n$.
According to the following proposition, the $k$-faces of the Weyl tessellation of type $A_{n-1}$
are exactly the cones of the above form different from $\{0\}$.

\begin{prop}\label{Prop_Char_Weyl_Faces_An-1}
Let $1\le k\le d$ and let $y_1,\ldots,y_n\in\R^d$ satisfy the general position assumption \ref{label_GP1_A}. Then it holds:
\begin{enumerate}[label=(\roman*)]
\item For every face $F\in\F^A_k(y_1,\ldots,y_n)$ there are a collection of indices $1\le l_1<\ldots<l_ {n-d+k-1}\le n-1$ and a permutation $\sigma\in\mathcal{S}_n$, such that $F=F^A_{\sigma}(l_1,\ldots,l_{n-d+k-1})$. Moreover, this representation is unique.
\item Let $1\le l_1<\ldots<l_{n-d+k-1}\le n-1$ and $\sigma\in\mathcal{S}_n$. If $F^A_{\sigma}(l_1,\ldots,l_{n-d+k-1})\neq \{0\}$, then $F^A_{\sigma}(l_1,\ldots,l_{n-d+k-1})\in \F^A_k(y_1,\ldots,y_n)$.
\item Every $k$-face $F^A_{\sigma}(l_1,\ldots,l_{n-d+k-1})\in\mathcal{F}^A_k(y_1,\ldots,y_n)$ is contained in exactly
\begin{align*}
l_1!(l_2-l_1)!\cdot\ldots\cdot(n-l_{n-d+k-1})!
\end{align*}
cones $C\in\mathcal W^A(y_1,\ldots,y_n)$.
\end{enumerate}
\end{prop}

\subsection{Reduction to linear subspaces intersecting Weyl chambers}
\label{sec:proof_faces_multiplicity_reduction}
Now, we want to reduce our problem of counting the faces in the Weyl tessellation in $\R^d$
to counting the number of faces of Weyl chambers in $\R^n$ intersected by a certain linear subspace.

\subsubsection*{Type $B_n$}
The \textit{reflection arrangement of type $B_n$}, denoted by $\A(B_n)$, consists of the hyperplanes given by
\begin{align}\label{Ey_Reflection_Arrangement_Bn}
&\{\beta\in\R^n:\beta_i=\beta_j\}\quad (1\le i<j\le n),\nonumber\\
&\{\beta\in\R^n:\beta_i=-\beta_j\}\quad (1\le i<j\le n),\\
&\{\beta\in\R^n:\beta_i=0\}\quad (1\le i\le n)\nonumber.
\end{align}
The \textit{closed Weyl chambers of type $B_n$} are the cones of the conical tessellation in $\R^n$
generated by $\A(B_n)$.   Thus, the closed Weyl chambers of type $B_n$ are given by
\begin{align*}
C^B_{\eps,\sigma}:=\{(\beta_1,\ldots,\beta_n)\in\R^n:\varepsilon_1\beta_{\sigma(1)}\le\ldots\le\varepsilon_n\beta_{\sigma(n)}\le 0\},\quad\sigma\in\mathcal{S}_n,\, \varepsilon\in\{\pm 1\}^n.
\end{align*}
Recall that $\mathcal S_n$ denotes the group of permutations of the set $\{1,\ldots,n\}$. The $k$-dimensional faces of the Weyl chamber $C^B_{\eps,\sigma}$ correspond to collections of indices $1\le l_1<\ldots<l_k\le n$ and have the form
\begin{align}\label{Eq_Rep_Weyl_Chamber_Faces_Bn}
C^B_{\eps,\sigma}(l_1,\ldots,l_k):=\{\beta\in\R^n&:\varepsilon_1\beta_{\sigma(1)}=\ldots=\varepsilon_{l_1}\beta_{\sigma(l_1)}\le \varepsilon_{l_1+1}\beta_{\sigma(l_1+1)}=\ldots=\varepsilon_{l_2}\beta_{\sigma(l_2)}\notag\\
&	\;\;\le\ldots\le \varepsilon_{l_{k-1}+1}\beta_{\sigma(l_{k-1}+1)}=\ldots=\varepsilon_{l_k}\beta_{\sigma(l_k)}\\
&	\;\;\le \beta_{\sigma(l_k+1)}=\ldots=\beta_{\sigma(n)}=0\}.\notag
\end{align}
In the case $l_k=n$, no $\beta_i$'s are required to be 0 (but the conditions $\varepsilon_i \beta_{\sigma(i)}\leq 0$ remain in force). In particular,  the number of $k$-faces of any Weyl chamber of type $B_n$ is given by $f_k(C^B_{\eps,\sigma})=\binom{n}{k}$.

Now, we need to introduce another general position condition, called~\ref{label_GP2}, which is equivalent to condition~\ref{label_GP1}. This equivalence is stated in the following theorem and is proved in Section \ref{Subsection_Proof_GeneralPosition}.

\begin{satz}\label{Theorem_Aequivalenz_B1_B2}
For arbitrary $y_1,\ldots,y_n\in\R^d$ the following conditions \ref{label_GP1} and \ref{label_GP2} are equivalent:
\begin{enumerate}[label=(B\arabic*), leftmargin=50pt]
\item For every $\varepsilon=(\varepsilon_1,\ldots,\varepsilon_n)\in\{\pm 1\}^n$ and $\sigma\in\mathcal{S}_n$ the vectors $\varepsilon_1y_{\sigma(1)}-\varepsilon_2y_{\sigma(2)},\varepsilon_2y_{\sigma(2)}-\varepsilon_3y_{\sigma(3)},\ldots,\varepsilon_{n-1}y_{\sigma(n-1)}-\varepsilon_ny_{\sigma(n)}, \varepsilon_ny_{\sigma(n)}$ are in general position, and $n\ge d$.
\item The linear subspace $L^\perp$ has dimension $d$ and is in general position with respect to the hyperplane arrangement $\mathcal{A}(B_n)$, where $L:=\{\beta\in\R^n:\beta_1y_1+\ldots+\beta_ny_n=0\}$, and $n\ge d$.\label{label_GP2}
\end{enumerate}
\end{satz}

The following lemma states the explicit connection between the faces of the Weyl tessellation $\mathcal W^B(y_1,\ldots,y_n)$ of type $B_n$ and the faces of the Weyl chambers of the same type.
Recall from \eqref{Eq_Rep_Faces_Bn} the notation $F^B_{\eps,\sigma}(l_1,\ldots,l_{n-d+k})$
for the $k$-faces of the Weyl tessellation of type $B_n$.

\begin{lem}\label{Lem_Connection_Face_WeylFace} Let $1\le k\le d$ and let $y_1,\ldots,y_n\in\R^d$ satisfy one of the equivalent general position assumptions \ref{label_GP1} or \ref{label_GP2}. Consider the linear subspace
$L=\{\beta\in\R^n:\beta_1y_1+\ldots+\beta_ny_n=0\}$. Then, the equivalence
\begin{align}\label{Eq_Lemma_Equivalence_Weylfaces}
F^B_{\eps,\sigma}(l_1,\ldots,l_{n-d+k})=\{0\}\Leftrightarrow C^B_{\eps,\sigma}(l_1,\ldots,l_{n-d+k})\cap L^\perp=\{0\}
\end{align}
holds true for all $1\le l_1<\ldots<l_{n-d+k}\le n$, $\eps\in\{\pm 1\}^n$ and $\sigma\in\mathcal{S}_n$.
\end{lem}

Before we prove this lemma, we want to state separately the special case $\eps_i=+1$, $\sigma(i)=i$ and $k=d$. In this case the lemma states that
\begin{align*}
\{v\in\R^d: \langle v,y_1\rangle\le\ldots\le \langle v,y_n\rangle\le 0\}=\{0\}\Leftrightarrow \{\beta\in\R^n:\beta_1\le\ldots\le\beta_n\le 0\}\cap L^\perp=\{0\}.
\end{align*}
This means that the Weyl cone on the left-hand side is degenerate, i.e.~ equal to $\{0\}$, if and only if the corresponding Weyl chamber, having the same order of inequalities, intersects $L^\perp$ only in a trivial way. Thus, using the general form of the  lemma, we can count the number of faces of the Weyl chambers intersected by $L^\perp$ in a non-trivial way instead of counting the faces of the Weyl tessellation.

To prove the lemma,  we need some additional results on general position. The following fact is standard; see, e.g.,\ \cite[p.~167]{Zeigler_LecturesOnPolytopes}.

\begin{lem}\label{Lemma_Positivkombination_Null}
Let $y_1,\ldots,y_n\in\R^d$ be in general position. Then $\pos\{y_1,\ldots,y_n\}=\R^d$ holds if and only if there exists an $\alpha=(\alpha_1,\ldots,\alpha_n)\in\R^n_{\ge 0}\setminus\{0\}$ such that $\alpha_1y_1+\ldots+\alpha_ny_n=0$.
\end{lem}

The following result is a slightly more general version of the Farkas lemma; see the proof of~\cite[Lemma~2.1]{HugSchneider2016} and use the orthogonal decomposition of $C$ into $\linsp (C)$ and $C\cap (\linsp (C))^\perp$.

\begin{lem}\label{Lemma_Cone_Subspace_LinSp}
Let $C$ be a cone and $L$ be a subspace in $\R^d$. Then
\begin{align*}
C\cap L\nsubseteq\linsp(C)\Leftrightarrow \relint(C^\circ)\cap L^\perp=\emptyset.
\end{align*}
\end{lem}

Additionally, we use a result on general position established in~\cite{KVZ15}. In our case, we need the following more general version.

\begin{lem}\label{Lemma_General_Position_Boundary_Faces}
Let $\A$ be a hyperplane arrangement in $\R^d$ and let $L$ be a linear subspace in $\R^d$. Furthermore let $\xbar{\mathcal R}(\cA)$ be the conical tessellation generated by $\cA$. If $L$ is in general position with respect to $\A$, then for all faces of the conical tessellation $F\in\mathcal{F}(\xbar{\mathcal{R}}(\A))$, it holds that
\begin{align*}
F\cap L\neq\{0\}\Leftrightarrow\relint(F)\cap L\neq\emptyset.
\end{align*}
\end{lem}

\begin{proof}
Let $\mathcal{A}=\{H_1,\ldots,H_n\}$. We define $G(i):=H_{i_1}\cap\ldots\cap H_{i_k}$ for all $k\le d$ and $1\le i_1<\ldots<i_k\le n$. Then, the linear subspace $L\cap G(i)$ is in general position with respect to the induced hyperplane arrangement $\mathcal{A}|G(i):= \{H_j \cap G(i): H_j \nsupseteq G(i), j=1,\ldots,n\}$.   Since every face $F\in\mathcal{F}(\xbar{\mathcal{R}}(\A))$ is contained in such a subspace $G(i)$ for a suitable collection of indices and is furthermore a cone of the induced tessellation $\xbar{\mathcal{R}}(\A|G(i))$, Lemma 3.5 from~\cite{KVZ15} applied to the linear subspace $L\cap G(i)$ in the ambient space $G(i)$ yields
\begin{align*}
F\cap L\neq\{0\}\Leftrightarrow F\cap (L\cap G(i))\neq\{0\}\Leftrightarrow \relint(F)\cap (L\cap G(i))\neq\emptyset\Leftrightarrow\relint(F)\cap L\neq\emptyset,
\end{align*}
which completes the proof.
\end{proof}

\begin{proof}[Proof of Lemma~\ref{Lem_Connection_Face_WeylFace}]
Without loss of generality we restrict ourselves to the special case $\sigma(i)=i$ and $\eps_i=1$,
for all $i=1,\ldots,n$. The general case follows by applying a suitable signed permutation of the coordinates.
For $1\le l_1<\ldots<l_{n-d+k}\le n$ we define
\begin{align*}
F:=\{v\in \R^d: f_1=\ldots=f_{l_1}\le f_{l_1+1}=\ldots =f_{l_2}\le\ldots\le f_{l_{n-d+k}+1}=\ldots=f_{n}=0\},
\end{align*}
which is just a shorthand notation for $F^B_{\eps,\sigma}(l_1,\ldots,l_{n-d+k})$ in the special case $\eps_i=1$ and $\sigma(i)=i$. Note that the linear functionals $f_i$ are defined by $f_i=f_i(v):=\langle v,y_i\rangle$ as above. We already saw in Proposition~\ref{Prop_Char_WeylFaces} that if $F$ is not $\{0\}$, it is a $k$-face of the Weyl tessellation $\mathcal{W}^B(y_1,\ldots,y_n)$, due to the general position assumption \ref{label_GP1}. Then, the linear hull of $F$ is given by
\begin{align*}
\lin F=\{v\in\R^d:f_1=\ldots=f_{l_1},f_{l_1+1}=\ldots=f_{l_2},\ldots,f_{l_{n-d+k}+1}=\ldots=f_n=0\},
\end{align*}
since the condition of $\lin F$ consists of $d-k$ independent equations, and therefore, the general position assumption \ref{label_GP1} implies that $\dim (\lin F)=d-(d-k)=k$ holds true. Using the identity $(C\cap D)^\circ=C^\circ+D^\circ$ which is valid for arbitrary polyhedral cones, see~\cite[Proposition~2.5]{Amelunxen2017},  together with  the well-known fact that $(\pos\{x_1,\ldots,x_n\})^\circ=\bigcap_{i=1}^n x_i^-$ for all $x_1,\ldots,x_n\in\R^d$, we get
\begin{align*}
F^\circ
&	=\bigg(\lin F\cap\bigcap_{i=1}^{n-d+k}(y_{l_i}-y_{l_i+1})^-\bigg)^\circ=(\lin F)^\circ+\bigg(\bigcap_{i=1}^{n-d+k}(y_{l_i}-y_{l_i+1})^-\bigg)	^\circ\\
&	=(\lin F)^\perp+\pos\{y_{l_1}-y_{{l_1}+1},\ldots,y_{l_{n-d+k}}-y_{{l_{n-d+k}}+1}\},
\end{align*}
where we set $y_{n+1}=0$. Thus, we get
\begin{align*}
F^\circ\cap \lin F=\pos\{\Pi(y_{l_1}-y_{{l_1}+1}),\ldots,\Pi(y_{l_{n-d+k}}-y_{{l_{n-d+k}}+1})\},
\end{align*}
where $\Pi:\R^d\to \lin (F)$ is the orthogonal projection onto $\lin (F)$. In order to prove this claim, we can represent the vectors $y_{l_i}-y_{l_i+1}$  as $(z_i,x_i)$, where $z_i$ is the projection on $\lin F$, and $x_i$ is the projection on  $(\lin F)^\perp$. Thus, the vectors in $F^\circ= (\lin F)^\perp+\pos\{y_{l_1}-y_{{l_1}+1},\ldots,y_{l_{n-d+k}}-y_{{l_{n-d+k}}+1}\}$ take the form
\begin{align*}
\begin{pmatrix}
\alpha_1z_1+\ldots + \alpha_{n-d+k}z_{n-d+k}\\
v+\alpha_1x_1+\ldots+\alpha_{n-d+k}x_{n-d+k}
\end{pmatrix}
\end{align*}
for $\alpha_1,\ldots,\alpha_{n-d+k}\ge 0$ and $v\in (\lin F)^\perp$. Here, the first entry denotes the component in $\lin F$ and the second entry denotes the component in $(\lin F)^\perp$. Such a vector is contained in $\lin F$ if and only if the second component is $0$, that is, $-v=\alpha_1x_1+\ldots+\alpha_{n-d+k}x_{n-d+k}$. Therefore, $F^\circ\cap \lin F = \pos(z_1,\ldots,z_{n-d+k})$.

Taking all of that into consideration, we obtain
\begin{align*}
F=\{0\}
&\Leftrightarrow F^\circ=\R^d\\
&\Leftrightarrow F^\circ\cap\lin F=\lin F\\
&\Leftrightarrow \pos\{\Pi(y_{l_1}-y_{{l_1}+1}),\ldots,\Pi(y_{l_{n-d+k}}-y_{{l_{n-d+k}}+1})\}=\lin F.
\end{align*}
Note that we used the decomposition $F^\circ=(\lin F)^\perp+(F^\circ\cap \lin F)$ for the second equivalence. Applying Lemma~\ref{Lemma_Positivkombination_Null} in the ambient $k$-dimensional linear subspace $\lin F$, we have that $\lin F=\pos\{\Pi(y_{l_1}-y_{{l_1}+1}),\ldots,\Pi(y_{l_{n-d+k}}-y_{{l_{n-d+k}}+1})\}$ if and only if there exist $\alpha_{l_1},\ldots\alpha_{l_{n-d+k}}\ge 0$ that do not vanish simultaneously and such that
\begin{align*}
0
&	=\alpha_{l_1}\Pi(y_{l_1}-y_{l_1+1})+\ldots+\alpha_{l_{n-d+k}}\Pi(y_{l_{n-d+k}}-y_{l_{n-d+k}+1})\\
&	=\Pi\big(\alpha_{l_1}(y_{l_1}-y_{l_1+1})+\ldots+\alpha_{l_{n-d+k}}(y_{l_{n-d+k}}-y_{l_{n-d+k}+1})\big).
\end{align*}
This holds if and only if there exists an $\alpha=(\alpha_1,\ldots,\alpha_n)\in\R^n$ with $\alpha_{l_1},\ldots,\alpha_{l_{n-d+k}}\ge 0$ not all being 0, such that
\begin{align*}
0=\alpha_1(y_1-y_2)+\ldots+\alpha_{n-1}(y_{n-1}-y_n)+\alpha_ny_n,
\end{align*}
since $(\lin (F))^\perp=\lin\{y_i-y_{i+1}:i\in\{1,\ldots,n\}\backslash\{l_1,\ldots,l_{n-d+k}\}\}$, $y_{n+1}:=0$.
After regrouping the terms, the condition takes  the form
\begin{align*}
0	=\alpha_1y_1+(\alpha_2-\alpha_1)y_2+\ldots+(\alpha_n-\alpha_{n-1})y_n.
\end{align*}
By defining $\beta_1=\alpha_1$, $\beta_i=\alpha_i-\alpha_{i-1}$ for $i=2,\ldots,n$, we see that this is equivalent to the existence of a vector $\beta\in\R^n$ with $\beta_1+\ldots+\beta_{l_1}\ge 0, \beta_1+\ldots+\beta_{l_2}\ge 0,\ldots,\beta_1+\ldots+\beta_{l_{n-d+k}}\ge 0$, where at least one inequality is strict, such that
\begin{align*}
0=\beta_1y_1+\ldots+\beta_{n}y_{n}.
\end{align*}
By defining $M:= \{\beta\in\R^n:\beta_1+\ldots+\beta_{l_1}\ge 0,\ldots,\beta_1+\ldots+\beta_{l_{n-d+k}}\ge 0\}$, recalling that $L=\{\beta\in\R^n:\beta_1y_1+\ldots+\beta_ny_n=0\}$ and taking the previous results into account, we get
\begin{align*}
F=\{0\}\Leftrightarrow M\cap L\nsubseteq\linsp (M),
\end{align*}
since $\linsp M:= M\cap(-M)=\{\beta\in\R^n:\beta_1+\ldots+\beta_{l_1}=0,\ldots,\beta_1+\ldots+\beta_{l_{n-d+k}}=0\}$ is the lineality space of $M$. Using Lemma~\ref{Lemma_Cone_Subspace_LinSp}, we get
\begin{align*}
M\cap L\nsubseteq\linsp (M)\Leftrightarrow\relint(M^\circ)\cap L^\perp=\emptyset.
\end{align*}
For the dual cone of $M$, the following holds:
\begin{align*}
M^\circ
&	=\big(\{\beta\in\R^n:\langle(\overbrace{1,\ldots,1}^{l_1},0,\ldots,0)^T,\beta\rangle\ge 0,\ldots,\langle(\overbrace{1,\ldots,1}^{l_{n-d+k}},0\ldots,0)^T,\beta\rangle\ge 0\}\big)^\circ\\
&	=-\pos\big\{(\overbrace{1,\ldots,1}^{l_1},0,\ldots,0)^T,\ldots,(\overbrace{1,\ldots,1}^{l_{n-d+k}},0\ldots,0)^T\big\}\\
&	=\{x\in\R^n:x_1=\ldots=x_{l_1}\le x_{l_1+1}=\ldots=x_{l_2}\le\ldots\le x_{l_{n-d+k}+1}=\ldots=x_n=0\}\\
&	=:G,
\end{align*}
where $G$ is just a shorthand notation for the $(n-d+k)$-dimensional face $C^B_{\eps,\sigma}(l_1,\ldots,l_{n-d+k})$ of the Weyl chambers of type $B_n$ in the special case $\eps_i=1$ and $\sigma(i)=i$.
Taking all equivalences into consideration, applying Lemma~\ref{Lemma_General_Position_Boundary_Faces} under the assumption of~\ref{label_GP2}, we get
\begin{align*}
F=\{0\}\Leftrightarrow\relint(G)\cap L^\perp=\emptyset\Leftrightarrow G\cap L^\perp=\{0\},
\end{align*}
which completes the proof of the equivalence (\ref{Eq_Lemma_Equivalence_Weylfaces}).
\end{proof}

\subsubsection*{Type $A_{n-1}$}

The $A_{n-1}$-case can be treated in a similar way. The \textit{closed Weyl chambers of type $A_{n-1}$} are the cones of the conical tessellation in $\R^n$ (not $\R^{n-1}$!) induced by the hyperplane arrangement $\A(A_{n-1})$ consisting of the hyperplanes given by
\begin{align}\label{Ey_Reflection_Arrangement_An-1}
\{\beta\in\R^n:\beta_i=\beta_j\}\quad (1\le i<j\le n).
\end{align}
It is called the \textit{reflection arrangement of type $A_{n-1}$}. Thus, the closed Weyl chambers of type $A_{n-1}$ are given by
\begin{align*}
C^A_{\sigma}:=\{(\beta_1,\ldots,\beta_n)\in\R^n:\beta_{\sigma(1)}\le\ldots\le\beta_{\sigma(n)}\},\quad \sigma\in\mathcal{S}_n.
\end{align*}
The $k$-dimensional faces of the Weyl chambers $C^A_{\sigma}$ are determined by  collections of indices $1\le l_1<\ldots<l_{k-1}\le n-1$ and have the form
\begin{multline}\label{Eq_Rep_Weyl_Chamber_Faces_An-1}
C^A_{\sigma}(l_1,\ldots,l_{k-1}):=\{\beta\in\R^n:\beta_{\sigma(1)}=\ldots=\beta_{\sigma(l_1)}\le \beta_{\sigma(l_1+1)}=\ldots=\beta_{\sigma(l_2)}\le\\\ldots\le \beta_{\sigma(l_{k-1}+1)}=\ldots=\beta_{\sigma(n)}\}.
\end{multline}
Thus, the number of $k$-faces of $C_\sigma^A$ is given by $f_k(C_\sigma^A)=\binom{n-1}{k-1}$.
The next theorem gives a general position assumption equivalent to~\ref{label_GP1_A}.
Its proof is similar to the proof of Theorem~\ref{Theorem_Aequivalenz_B1_B2}; see Section~\ref{Subsection_Proof_GeneralPosition}.

\begin{satz}\label{Theorem_Aequivalenz_A1_A2}
For arbitrary $y_1,\ldots,y_n\in\R^d$ the following conditions \ref{label_GP1_A} and \ref{label_GP2_A} are equivalent:
\begin{enumerate}[label=(A\arabic*), leftmargin=50pt]
\item For every  $\sigma\in\mathcal{S}_n$ the vectors $y_{\sigma(1)}-y_{\sigma(2)},y_{\sigma(2)}-y_{\sigma(3)},\ldots,y_{\sigma(n-1)}-y_{\sigma(n)}$ are in general position, and $n\ge d+1$.
\item The linear subspace $L^\perp$ has dimension $d$ and is in general position with respect to the hyperplane arrangement $\mathcal{A}(A_{n-1})$, where $L:=\{\beta\in\R^n:\beta_1y_1+\ldots+\beta_ny_n=0\}$, and $n\ge d+1$.\label{label_GP2_A}
\end{enumerate}
\end{satz}

 The following lemma is an analogue to Lemma~\ref{Lem_Connection_Face_WeylFace}. Recall the notation $F^A_\sigma(l_1,\ldots,l_{n-d+k-1})$ for the $k$-faces of the Weyl tessellation of type $A_{n-1}$ from \eqref{Eq_Rep_Weyl_Faces_An-1}.

\begin{lem}\label{Lem_Connection_Faces_Chamber_An-1}
Let $1\le k\le d$ and let $y_1,\ldots,y_n\in\R^d$ satisfy one of the equivalent general position assumptions \ref{label_GP1_A} or \ref{label_GP2_A}. Define  $L=\{\beta\in\R^n:\beta_1y_1+\ldots+\beta_ny_n=0\}$. The equivalence
\begin{align*}
F_{\sigma}^A(l_1,\ldots,l_{n-d+k-1})=\{0\}\Leftrightarrow C^A_{\sigma}(l_1,\ldots,l_{n-d+k-1})\cap L^\perp=\{0\}
\end{align*}
holds true for all $1\le l_1<\ldots<l_{n-d+k-1}\le n-1$ and $\sigma\in\mathcal{S}_n$.
\end{lem}


\begin{proof}
The proof is similar to that of Lemma~\ref{Lem_Connection_Face_WeylFace} and we will not explain each argument in full detail.
Let $1\le l_1<\ldots<l_{n-d+k-1}\le n-1$.  Applying a permutation of the coordinates, we may restrict ourselves to the special case $\sigma(i)=i$ for all $i\in \{1,\ldots,n\}$. Define
\begin{align*}
F=\{v\in\R^d:f_1=\ldots= f_{l_1}\le f_{l_1+1}=\ldots =f_{l_2}\le\ldots\le f_{l_{n-d+k-1}+1}=\ldots=f_n\},
\end{align*}
where the $f_i$'s are the functionals defined by $f_i=f_i(v):=\langle v,y_i\rangle$. Then, we have
\begin{align*}
F=\{0\}&\Leftrightarrow F^\circ\cap\lin F=\lin F\\
&\Leftrightarrow\pos\{\Pi(y_{l_1}-y_{{l_1}+1}),\ldots,\Pi(y_{l_{n-d+k-1}}-y_{l_{n-d+k-1}+1})\}=\lin F,
\end{align*}
where $\Pi$ denotes the orthogonal projection onto $\lin F$. Using Lemma~\ref{Lemma_Positivkombination_Null} and the general position assumption \ref{label_GP1_A}, this is equivalent to the fact that there are $\alpha_{l_1},\ldots\alpha_{l_{n-d+k-1}}\ge 0$ that do not vanish simultaneously and satisfy
\begin{align*}
0
&
	=\Pi\big(\alpha_{l_1}(y_{l_1}-y_{l_1+1})+\ldots+\alpha_{l_{n-d+k-1}}(y_{l_{n-d+k-1}}-y_{l_{n-d+k-1}+1})\big).
\end{align*}
In view of
\begin{align*}
(\lin F)^\perp=\lin\{y_1-y_2,&\ldots,y_{l_1-1}-y_{l_1},\ldots,y_{l_{n-d+k-1}+1}-y_{l_{n-d+k-1}+2},\ldots,y_{n-1}-y_n\},
\end{align*}
the above is equivalent to the existence of a vector $(\alpha_1,\ldots,\alpha_{n-1})\in\R^{n-1}$, where $\alpha_{l_1},\ldots,\alpha_{l_{n-d+k-1}}\ge 0$ do not vanish simultaneously, such that
\begin{align*}
0	=\alpha_1(y_1-y_2)+\ldots+\alpha_{n-1}(y_{n-1}-y_n).
\end{align*}
After regrouping the terms, the condition takes the form
\begin{align*}
0	=\alpha_1y_1+(\alpha_2-\alpha_1)y_2+\ldots+(\alpha_{n-1}-\alpha_{n-2})y_{n-1}+(-\alpha_{n-1})y_n.
\end{align*}
By setting $\beta_1:=\alpha_1,\beta_i:=\alpha_i-\alpha_{i-1}$, $2\le i\le n-1$ and $\beta_n:=-\alpha_{n-1}=-(\beta_1+\ldots+\beta_{n-1})$, this is equivalent to the fact that there exists a vector $(\beta_1,\ldots,\beta_n)\in\R^n$ satisfying $\beta_1+\ldots+\beta_n=0$ and
\begin{align*}
\beta_1+\ldots+\beta_{l_1}\ge 0,\beta_1+\ldots+\beta_{l_2}\ge 0,\ldots,\beta_{1}+\ldots+\beta_{l_{n-d+k-1}}\ge 0,
\end{align*}
where at least one inequality is strict, such that $\beta_1y_1+\ldots+\beta_ny_n=0$. Similarly to the proof of Lemma~\ref{Lem_Connection_Face_WeylFace}, we define $M:=\{\beta\in\R^n:\beta_1+\ldots+\beta_{l_1}\ge 0,\ldots,\beta_1+\ldots+\beta_{l_{n-d+k-1}}\ge 0,\beta_1+\ldots+\beta_n=0\}$ and obtain
\begin{align*}
F=\{0\}\Leftrightarrow L\cap M\nsubseteq \linsp(M)\Leftrightarrow\relint(M^\circ)\cap L^\perp=\emptyset.
\end{align*}
Observe that
\begin{align*}
M^\circ
&	=\{x\in\R^d:x_1=\ldots=x_{l_1}\le x_{l_1+1}=\ldots=x_{l_2}\le\ldots\le x_{l_{n-d+k-1}+1}=\ldots=x_n\}.
\end{align*}
Note that $M^\circ$ coincides with the $k$-face $D^A_\sigma(l_1,\ldots,l_{n-d+k-1})$ of the Weyl tessellation of type $A_{n-1}$, in the special case $\sigma(i)=i$. Using Lemma~\ref{Lemma_General_Position_Boundary_Faces} and \ref{label_GP2_A}, we get
\begin{align*}
F=\{0\}\Leftrightarrow\relint(M^\circ)\cap L^\perp=\emptyset\Leftrightarrow M^\circ \cap L^\perp=\{0\},
\end{align*}
which proves  Lemma~\ref{Lem_Connection_Faces_Chamber_An-1}.
\end{proof}

\subsection{Proof of Propositions~\ref{prop:faces_with_multiplicity_A} and~\ref{prop:faces_with_multiplicity_B}}
\label{sec:proof_faces_multiplicity_final_part}

In the previous section, we reduced our problem to counting faces of Weyl chambers that are intersected by a linear subspace. A formula for this quantity is stated in the following two theorems. They were proven in \cite[Theorems~2.1 and~2.8]{Kabluchko2019}.

\begin{satz}\label{Theorem_Weyl_Faces_Intersected}
Let $L_d\in G(n,d)$ be a deterministic $d$-dimensional subspace of $\R^n$ in general position with respect to the reflection arrangement $\mathcal{A}(B_n)$. Then
\begin{align*}
\sum_{\eps\in\{\pm 1\}^n}\:\sum_{\sigma\in\mathcal{S}_n}\:\sum_{F\in\F_k(C^B_{\eps,\sigma})}\1_{\{F\cap L_d\neq\{0\}\}}
&	\quad=2^{n-k+1}\binom{n}{k}\frac{n!}{k!}\big(\stirlingb k{n-d+1}+\stirlingb k{n-d+3}+\ldots\big)\\
&	\quad=2^{n-k}\binom{n}{k}\frac{n!}{k!}D^B(k,d-n+k),
\end{align*}
where the $\stirlingb kj$'s are defined in~\eqref{eq:def_stirling1b}.
\end{satz}

In~\cite[Theorem 2.1]{Kabluchko2019}, the formula was stated for the complementary quantity counting the
faces that intersect the subspace $L_d$ only in a trivial way.
To derive the formula as stated above, one uses the identity
\begin{align*}
\stirlingb k1+\stirlingb k3+\ldots=\stirlingb k0+\stirlingb k2+\ldots=2^{k-1}k!
\end{align*}
(which follows from~\eqref{eq:def_stirling1b} by taking  $t=\pm 1$) together with the fact that the total  number
of Weyl chambers is $2^nn!$ and each chamber has $\binom{n}{k}$ faces of dimension $k$.
The analogous result in the $A_{n-1}$-case reads as follows.

\begin{satz}\label{Theorem_Weyl_Faces_Intersected_Type_An-1}
Let $L_d\in G(n,d)$ be a deterministic $d$-dimensional subspace of $\R^n$ in general position with respect to the reflection arrangement $\mathcal{A}(A_{n-1})$. Then
\begin{align*}
\sum_{\sigma\in\mathcal{S}_n}\:\sum_{F\in\F_k(C^A_{\sigma})}\1_{\{F\cap L_d\neq\{0\}\}}
&	=\frac{2n!}{k!}\binom{n-1}{k-1}\bigg(\stirling{k}{n-d+1}+\stirling{k}{n-d+3}+\ldots\bigg)\\
&	=\frac{n!}{k!}\binom{n-1}{k-1}D^A(k,n-d+k),
\end{align*}
where the $\stirling{n}{k}$'s are the Stirling numbers of the first kind as defined in~\eqref{eq:def_stirling1}.
\end{satz}

\begin{proof}[Proof of Proposition~\ref{prop:faces_with_multiplicity_B}]
Let $1\le k\le d$ and let $y_1,\ldots,y_n$ satisfy one of the equivalent general position assumptions \ref{label_GP1} or \ref{label_GP2}. We want to evaluate the number of $k$-faces of $F\in\F^B_k(y_1,\ldots,y_n)$, each face counted with the multiplicity equal to the number of $d$-dimensional Weyl cones $C\in\mathcal W^B(y_1,\ldots,y_n)$ containing it. We define $\Omega_n(l):=\{\eps\in\{\pm 1\}^n: \eps_{l_{n-d+k}+1}=\ldots=\eps_n=1\}$ and use Proposition~\ref{Prop_Char_WeylFaces} and Proposition~\ref{Prop_WeylKegel_die_Seiten_enthalten_B_n} to obtain
\begin{align*}
&\sum_{F\in\F^B_k(y_1,\ldots,y_n)}\:\sum_{C\in\mathcal W^B(y_1,\ldots,y_n)}\1_{\{F\subseteq C\}}\\
&	\quad=\sum_{1\le l_1<\ldots<l_{n-d+k}\le n}l_1!(l_2-l_1)!\ldots (n-l_{n-d+k})!2^{n-l_{n-d+k}}\sum_{\eps\in\Omega_n(l)}\sum_{\sigma\in\mathcal{S}_n}\1_{\{F^B_{\varepsilon,\sigma}(l_1,\ldots,l_{n-d+k})\neq\{0\}\}}\\
&	\quad=\sum_{1\le l_1<\ldots<l_{n-d+k}\le n}l_1!(l_2-l_1)!\ldots(n-l_{n-d+k})!2^{n-l_{n-d+k}}\sum_{\eps\in\Omega_n(l)}\sum_{\sigma\in\mathcal{S}_n}\1_{\{C^B_{\eps,\sigma}(l_1,\ldots,l_{n-d+k})\cap L^\perp\neq\{0\}\}}.
\end{align*}
Note that we applied the equivalence (\ref{Eq_Lemma_Equivalence_Weylfaces}) from Lemma~\ref{Lem_Connection_Face_WeylFace} in the last equation. Now, we use that each face $C_{\eps,\sigma}^B(l_1,\ldots,l_{n-d+k})$ is contained in exactly $l_1!(l_2-l_1)!\ldots(n-l_{n-d+k})!2^{n-l_{n-d+k}}$ Weyl chambers $C_{\eps,\sigma}^B$ of type $B_n$. For this standard fact, we refer to \cite[Proof of Theorem 2.1]{Kabluchko2019}. Furthermore $L^\perp$ is a $d$-dimensional subspace and in general position with respect to the reflection arrangement $\A(B_n)$, due to \ref{label_GP2}. Thus, we can apply Theorem~\ref{Theorem_Weyl_Faces_Intersected} replacing  $k$ by $n-d+k$ and get
\begin{align*}
\sum_{F\in\F^B_k(y_1,\ldots,y_n)}\:\sum_{C\in\mathcal W^B(y_1,\ldots,y_n)}\1_{\{F\subseteq C\}}
&	=\sum_{\eps\in\{\pm 1\}^n}\:\sum_{\sigma\in\mathcal{S}_n}\:\sum_{F\in \F_{n-d+k}(C^B_{\eps,\sigma})}\1_{\{F\cap L^\perp\neq\{0\}\}}\\
&	=2^{d-k}\binom{n}{d-k}\frac{n!}{(n-d+k)!}D^B(n-d+k,k),
\end{align*}
which completes the proof.
\end{proof}




\begin{proof}[Proof of Proposition~\ref{prop:faces_with_multiplicity_A}]
The proof is similar to that of Proposition~\ref{prop:faces_with_multiplicity_B}. Using Proposition~\ref{Prop_Char_Weyl_Faces_An-1} and Lemma~\ref{Lem_Connection_Faces_Chamber_An-1}, we get
\begin{align*}
&\sum_{F\in\F^A_k(y_1,\ldots,y_n)}\:\sum_{C\in\mathcal W^A(y_1,\ldots,y_n)}\1_{\{F\subseteq C\}}\\
&	\quad=\sum_{1\le l_1<\ldots<l_{n-d+k-1}\le n-1}l_1!(l_2-l_1)!\ldots (n-l_{n-d+k-1})!\sum_{\sigma\in\mathcal{S}_n}\1_{\{F^A_{\sigma}(l_1,\ldots,l_{n-d+k-1})\neq\{0\}\}}\\
&	\quad=\sum_{1\le l_1<\ldots<l_{n-d+k-1}\le n-1}l_1!(l_2-l_1)!\ldots(n-l_{n-d+k-1})!\sum_{\sigma\in\mathcal{S}_n}\1_{\{C^A_{\sigma}(l_1,\ldots,l_{n-d+k-1})\cap L^\perp\neq\{0\}\}}\\
&	\quad=\sum_{\sigma\in\mathcal{S}_n}\:\sum_{F\in \F_{n-d+k}(C^B_{\eps,\sigma})}\1_{\{F\cap L^\perp\neq\{0\}\}}\\
&	\quad=\binom{n-1}{d-k}\frac{n!}{(n-d+k)!}D^A(n-d+k,k),
\end{align*}
which completes the proof. Here, we used that each face $C_\sigma^A(l_1,\ldots,l_{n-d+k-1})$ is contained in exactly $l_1!(l_2-l_1)!\ldots(n-l_{n-d+k-1})!$ Weyl chambers $C_\sigma^A$; see, e.g.,  \cite[Proof of Theorem 2.8]{Kabluchko2019}. Since \ref{label_GP2_A} is satisfied, we were able to apply Theorem~\ref{Theorem_Weyl_Faces_Intersected_Type_An-1} in the last step.
\end{proof}

\section{Number of faces in Weyl tessellations: Proof of Theorems~\ref{Theorem_Number_WeylFaces_Tessellation_An-1} and~\ref{Theorem_Number_WeylFaces_Tessellation_Bn}}\label{sec:proof_weyl_tess}

This section contains the proofs of the formulas for the total number of $k$-faces in Weyl tessellations of both types $B_n$ and $A_{n-1}$.
For $k\in\{1,\ldots,d\}$, we want to prove that
\begin{align*}
\#\cF_k^B(y_1,\ldots,y_n)=\stirlingsecb n{n-d+k}D^B(n-d+k,k)
\end{align*}
and
\begin{align*}
\#\F_k^A(y_1,\ldots,y_n)=\stirlingsec{n}{n-d+k}D^A(n-d+k,k).
\end{align*}

\begin{proof}[Proof of Theorem~\ref{Theorem_Number_WeylFaces_Tessellation_Bn}]
Due to Proposition~\ref{Prop_Char_WeylFaces}(i), each $k$-face $F\in\F_k^B(y_1,\ldots,y_n)$ is contained in a $k$-dimensional  linear  subspace of the form
\begin{align*}
L(l,\eps,\sigma):=\{v\in\R^d&:\eps_1f_{\sigma(1)}=\ldots=\eps_{l_1}f_{\sigma(l_1)},\ldots,\\
&	\;\;\;\eps_{l_{n-d+k-1}+1}f_{\sigma(l_{n-d+k-1}+1)}=\ldots=\eps_{l_{n-d+k}}f_{\sigma(l_{n-d+k})},\\
&	\;\;\; f_{\sigma(l_{n-d+k}+1)}=\ldots=f_{\sigma(n)}=0\}
\end{align*}
for some $1\le l_1<\ldots<l_{n-d+k}\le n$, $l= (l_1,\ldots,l_{n-d+k})$,  $\eps\in\{\pm 1\}^n$ and $\sigma\in\mathcal{S}_n$. At first, we want to evaluate the number of distinct subspaces of this form. For each fixed size $r\in\{0,\ldots,d-k\}$ of the last group of equations there are $\binom{n}{r}$ possibilities to choose its elements among $\{f_1,\ldots,f_n\}$. Then, we are left with a set of $n-r$ elements, which we want to partition in $n-d+k$ non-empty subsets. There are $\stirlingsec{n-r}{n-d+k}$ possibilities to choose the partition. Furthermore, we can choose the signs of the $f_i$'s in the first $n-d+k$ groups arbitrarily, for which there are $2^{n-r}$ possibilities. But since we obtain the same subspace if we multiply any group of equations by $-1$, we have to divide the $2^{n-r}$ possibilities by $2^{n-d+k}$. This yields a total of
\begin{align*}
\sum_{r=0}^{d-k}\binom{n}{r}\stirlingsec{n-r}{n-d+k}\frac{2^{n-r}}{2^{n-d+k}}=\sum_{j=n-d+k}^{n}\binom{n}{j}\stirlingsec{j}{n-d+k}2^{j-(n-d+k)}=\stirlingsecb n{n-d+k}
\end{align*}
possible subspaces of the form $L(l,\eps,\sigma)$. The last equation follows from the definition of $\stirlingsecb nk$ in~\eqref{eq:def_stirling2b}. All these subspaces are pairwise distinct, which can be shown in the same way as in Example~\ref{Example_WeylCones_containing_face} and relies on the general position assumption~\ref{label_GP1}.

Now, we want to show that the $k$-faces of $\mathcal{W}^B(y_1,\ldots,y_n)$ contained in $L(l,\eps,\sigma)$ form a Weyl tessellation in $L(l,\eps,\sigma)$ and that the number of these $k$-faces is $D^B(n-d+k,k)$, independently of the choices of $l,\eps$ and $\sigma$.  To simplify the notation, we consider the special case $\eps_i=1$ and $\sigma(i)=i$, for all $i=1,\ldots,n$, and define
\begin{align*}
L:=\{v\in\R^d: f_1=\ldots=f_{l_1},\ldots,f_{l_{n-d+k-1}+1}=\ldots=f_{l_{n-d+k}},f_{l_{n-d+k}+1}=\ldots=f_n=0\},
\end{align*}
where  $f_i(v)= \langle v, y_i\rangle$.
Our goal is to show that the orthogonal projections $\Pi_L(y_{l_1}),\ldots,\Pi_L(y_{l_{n-d+k}})$ on $L$ induce a Weyl tessellation (of type $B_{n-d+k}$) in $L$ and that its $k$-dimensional cones are in one-to-one correspondence with the $k$-faces of $\mathcal W^B(y_1,\dots,y_n)$ contained in $L$. Postponing the verification of~\ref{label_GP1} for these projections to the end of the proof, Theorem~\ref{Theorem_Number_Weyl_Cones} implies that the Weyl tessellation in $L$ generated by $\Pi_L(y_{l_1}),\ldots,\Pi_L(y_{l_{n-d+k}})$ consists of $D^B(n-d+k,k)$ cones. These are the cones  different from $\{0\}$ of the form
\begin{align*}
&	\{v\in L:\delta_1\langle v,\Pi_L(y_{l_{\pi(1)}})\rangle\le\ldots\le\delta_{n-d+k}\langle v,\Pi_L(y_{l_{\pi(n-d+k)}})\rangle\le 0\}\\
&	\quad =\{v\in L:\delta_1f_{l_{\pi(1)}}\le \ldots\le\delta_{n-d+k}f_{l_{\pi(n-d+k)}}\le 0\}\\
&	\quad =\{v\in\R^d: \delta_1f_{l_{\pi(1)-1}+1}=\ldots=\delta_1f_{l_{\pi(1)}}\le \delta_2f_{l_{\pi(2)-1}+1}=\ldots=\delta_2f_{l_{\pi(2)}}\\
&	\quad\quad\quad\quad\quad\quad\;\le \ldots\le \delta_{n-d+k}f_{l_{\pi(n-d+k)-1}+1}=\ldots=\delta_{n-d+k}f_{l_{\pi(n-d+k)}}\le f_{l_{n-d+k}+1}=\ldots=f_n= 0\},
\end{align*}
where $\delta\in\{\pm 1\}^{n-d+k}$, $\pi\in\mathcal S_{n-d+k}$. Note that we used $\langle v,\Pi_L(y_{l_{\pi(i)}})\rangle=\langle v,y_{l_{\pi(i)}}\rangle$, for all $v\in L$ and $i=1,\ldots,n-d+k$, in the first equality. The second equality follows from the definition of $L$.   The last representation basically says that we keep the groups of equations from $L$, permute them  (except for the last one) according to $\pi$ and change the signs in the  groups  (except for the last one) according to $\delta$. Since we are interested only in cones different from $\{0\}$, the above representations define $k$-faces of $\mathcal{W}^B(y_1,\ldots,y_n)$ due to Proposition~\ref{Prop_Char_WeylFaces}(ii). Since the cones different from $\{0\}$ cover  $L$, these are already all of the $k$-faces contained in $L$.

In summary, we know that every $k$-face of $\mathcal{W}^B(y_1,\ldots,y_n)$ is contained in a  unique subspace of the form $L(l,\eps,\sigma)$ and every such subspace contains $D^B(n-d+k,k)$ faces of dimension $k$. This yields a total of
$
\stirlingsecb n{n-d+k} D^B(n-d+k,k)
$
$k$-faces of $\mathcal{W}^B(y_1,\ldots,y_n)$.

It remains to prove that $\Pi_L(y_{l_1}),\ldots,\Pi_L(y_{l_{n-d+k}})$ satisfy the general position assumption~\ref{label_GP1} in $L$, that is, the vectors
\begin{align*}
\delta_1 \Pi_L(y_{l_{\pi(1)}})- \delta_2\Pi_L(y_{l_{\pi(2)}}),\dots ,\delta_{n-d+k-1} \Pi_L(y_{l_{\pi(n-d+k-1)}})- \delta_{n-d+k}\Pi_L(y_{l_{\pi(n-d+k)}}),  \Pi_L(y_{l_{\pi(n-d+k)}})
\end{align*}
are in general position, for each $\delta\in\{\pm 1\}^{n-d+k}$ and $\pi\in\mathcal S_{n-d+k}$. So, let arbitrary $\delta\in\{\pm 1\}^{n-d+k}$ and $\pi\in\mathcal S_{n-d+k}$ be given. Define the hyperplanes arrangements
\begin{align*}
\cA_1\hspace*{-2pt}=\hspace*{-2pt}\Big\{
&(\delta_1y_{l_{\pi(1)-1}+1}-\delta_1y_{l_{\pi(1)-1}+2})^\perp,\dots,(\delta_1y_{l_{\pi(1)-1}}-\delta_1y_{l_{\pi(1)}})^\perp,\\
&(\delta_2y_{l_{\pi(2)-1}+1}-\delta_2y_{l_{\pi(2)-1}+2})^\perp,\dots,(\delta_2y_{l_{\pi(2)-1}}-\delta_2y_{l_{\pi(2)}})^\perp,\\
&\qquad\qquad\qquad\qquad\qquad\;\:\dots\qquad\qquad\qquad\qquad\qquad\qquad,\\
&(\delta_{n-d+k}y_{l_{\pi({n-d+k})-1}+1}-\delta_{n-d+k}y_{l_{\pi({n-d+k})-1}+2})^\perp,\dots,(\delta_{n-d+k}y_{l_{\pi({n-d+k})-1}}-\delta_{n-d+k}y_{l_{\pi({n-d+k})}})^\perp\hspace*{-1pt},\\
&(y_{l_{n-d+k}+1}-y_{l_{n-d+k}+1})^\perp,\dots,(y_{n-1}-y_{n})^\perp,y_n^\perp\Big\}
\end{align*}
and
\begin{multline*}
\cA_2=\Big\{(\delta_1y_{l_{\pi(1)}}-\delta_2y_{l_{\pi(1)}+1})^\perp,\dots,(\delta_{n-d+k-1}y_{l_{\pi(n-d+k-1)}}-\delta_{n-d+k}y_{l_{\pi(n-d+k-1)}+1})^\perp,\\(\delta_{n-d+k}y_{l_{\pi(n-d+k)}}-y_{l_{n-d+k}+1})^\perp\Big\}.
\end{multline*}
Assumption~\ref{label_GP1} implies that the hyperplanes in $\cA_1\cup\cA_2$ are in general position. Since $L$ is just the intersection of the hyperplanes from $\cA_1$, we also know that $L$ is in general position to $\cA_2$. Consequently,  the hyperplanes in $L$ of the induced arrangement
\begin{multline*}
\cA_2|L =\{H\cap L:H\in\cA_2\}\\
=\Big\{L\cap (\delta_1y_{l_{\pi(1)}}-\delta_2y_{l_{\pi(2)}})^\perp,\dots,L\cap(\delta_{n-d+k-1}y_{l_{\pi(n-d+k-1)}}-\delta_{n-d+k}y_{l_{\pi(n-d+k)}})^\perp,
L\cap y_{l_{\pi(n-d+k)}}^\perp\Big\}
\end{multline*}
are in general position in $L$. Since $\delta$ and $\pi$ were chosen arbitrarily, this implies that the  projected vectors $\Pi_L(y_{l_1}),\dots,\Pi_L(y_{l_{n-d+k}})$ satisfy~\ref{label_GP1} in the ambient subspace $L$.
\end{proof}

\begin{proof}[Proof of Theorem~\ref{Theorem_Number_WeylFaces_Tessellation_An-1}]
Due to Proposition~\ref{Prop_Char_Weyl_Faces_An-1}(i), each $k$-face of $F\in\F_k^A(y_1,\ldots,y_n)$ is contained in a subspace of the form
\begin{align*}
L(l,\sigma):=\{v\in\R^d:f_{\sigma(1)}=\ldots=f_{\sigma(l_1)},\ldots,f_{\sigma(l_{n-d+k-1}+1)}=\ldots=f_{\sigma(n)}\}
\end{align*}
for suitable $1\le l_1<\ldots<l_{n-d+k-1}\le n-1$ and $\sigma\in\mathcal{S}_n$. There are a total of $\stirlingsec{n}{n-d+k}$ distinct subspaces of the given form, since these are in one-to-one correspondence with partitions of the set $\{1,\ldots,n\}$ into $n-d+k$ non-empty sets.

Now, it is left to prove that every subspace $L(l,\sigma)$ contains exactly $D^A(n-d+k,k)$ $k$-faces of $\mathcal{W}^A(y_1,\ldots,y_n)$. Again, consider only the case $L:=L(l,\sigma)$ for $\sigma(i)=i$, $i=1,\ldots,n$. For this, we need to show that $\Pi_L(y_{l_1}),\ldots,\Pi_L(y_{l_{n-d+k-1}}),\Pi_L(y_n)$ satisfy the general position assumption \ref{label_GP1_A} in $L$, which is shown in the same way as in Theorem~\ref{Theorem_Number_WeylFaces_Tessellation_Bn}. This completes the proof.
\end{proof}

\section{Expected size functionals of Weyl random cones: Proof of Theorems~\ref{Theorem_sizefunction_Dn_An-1} and~\ref{Theorem_sizefunction_Dn}}\label{sec:proofs_size_funct}

This section is dedicated to proving the formulas for the expected size functionals of the Weyl random cones $\Dna$ and $\Dnb$ stated in Section~\ref{sec:weyl_random_cones}.

\subsection{Characterizing  faces induced in a linear subspace}

In order to prove Theorems~\ref{Theorem_sizefunction_Dn_An-1} and~\ref{Theorem_sizefunction_Dn}, we need to state a result on the faces of the Weyl tessellations induced in a linear subspace. For this, we introduce the following notation in the $B_n$-case. Let $U\subseteq\R^d$ be a $k$-dimensional linear subspace in $\R^d$ which is in general position with respect to  $\mathcal{A}^B(y_1,\ldots,y_n)$. Recall that the hyperplane arrangement induced by $\A^B(y_1,\ldots,y_n)$ in $U$ is defined as $\A^B|_U(y_1,\ldots,y_n):=\{H\cap U:H\in \A^B(y_1,\ldots,y_n)\}$. The induced arrangement $\A^B|_U(y_1,\ldots,y_n)$ consists explicitly of the following hyperplanes in $U$:
\begin{align*}
 \big(\Pi_U(y_i)+\Pi_U(y_j)\big)^\perp \cap U,\quad &1 \le i<j\le n,\\
\big(\Pi_U(y_i)-\Pi_U(y_j)\big)^\perp\cap U,\quad &1 \le i<j\le n,\\
\Pi_U(y_i)^\perp\cap U,\quad &1\le i\le n.
\end{align*}
By definition, the induced Weyl tessellation in $U$, which we will denote by $\mathcal{W}^B|_U(y_1,\ldots,y_n)$, consists of  the cones of the conical tessellation  in $U$ generated by the hyperplane arrangement $\A^B|_U(y_1,\ldots,y_n)$. We denote the set of $j$-faces of $\mathcal{W}^B|_U(y_1,\ldots,y_n)$ by $\F^B_j|_U(y_1,\ldots,y_n)$. To state an explicit representation of these faces,  we define the cones
$$
F^B_{\eps,\,\sigma}|_U(l_1,\ldots,l_{n-k+j})
:=
F_{\eps,\,\sigma}^B(l_1,\ldots,l_{n-k+j})\cap U,
$$
where  $\eps\in\{\pm 1\}^n$, $\sigma\in\mathcal{S}_n$, $1\le j\le k\le d$, and $1\leq l_1<\ldots <l_{n-k+j}\leq n$.
Let $f_i'$ be the linear functionals on $U$ given by $f'_i=f'_i(v):=\langle v,\Pi_U(y_i)\rangle$ for $i=1,\ldots,n$. Since
\begin{align*}
f_i(v)=\langle v,y_i\rangle=\langle v,\Pi_U(y_i)\rangle+\langle v, y_i-\Pi_U(y_i)\rangle=\langle v,\Pi_U(y_i)\rangle=f'_i(v)
\end{align*}
holds for all $v\in U$, we have the explicit representation
\begin{align}
F^B_{\eps,\,\sigma}|_U(l_1,\ldots,l_{n-k+j})
=
\big\{v\in U&:\varepsilon_1f'_{\sigma(1)}=\ldots=\varepsilon_{l_1}f'_{\sigma(l_1)}\le \varepsilon_{l_1+1}f'_{\sigma(l_1+1)}=\ldots =\varepsilon_{l_2}f'_{\sigma(l_2)}\notag\\
&	\;\;\le\ldots\le \varepsilon_{l_{n-k+j-1}+1}f'_{\sigma(l_{n-k+j-1}+1)}= \ldots=\varepsilon_{l_{n-k+j}}f'_{\sigma(l_{n-k+j})} \label{eq:F_B_U_rep}\\
&	\;\;\le f'_{\sigma(l_{n-k+j}+1)}=\ldots=f'_{\sigma(n)}=0\big\}. \notag
\end{align}
We will see below that if not $\{0\}$, the cones  $F^B_{\eps,\,\sigma}|_U(l_1,\ldots,l_{n-k+j})$ are the $j$-faces of the induced Weyl tessellation $\mathcal{W}^B|_U(y_1,\ldots,y_n)$.

\begin{lem}\label{Lemma_Faces_Induced_in_L}
Let $1\le j\le k\le d$ and let $y_1,\ldots,y_n\in\R^d$ satisfy the general position assumption \ref{label_GP1}. Furthermore, let $U\in G(d,k)$ be in general position with respect to the hyperplane arrangement $\mathcal{A}^B(y_1,\ldots,y_n)$. Then the following hold:
\begin{enumerate}[label=(\roman*)]
\item For every $j$-face $F_j\in \F_j^B|_U(y_1,\ldots,y_n)$ of the tessellation $\mathcal{W}^B|_U(y_1,\ldots,y_n)$ there is a unique $(d-k+j)$-face $F\in\F^B_{d-k+j}(y_1,\ldots,y_n)$ containing $F_j$ and satisfying $F_j=F\cap U$.
\item If $F\in\F^B_{d-k+j}(y_1,\ldots,y_n)$ and $F\cap U\neq \{0\}$, then $F\cap U\in\F_j^B|_U(y_1,\ldots,y_n)$.
\end{enumerate}
\end{lem}

\begin{proof} 
At first, we show that the projections $\Pi_U(y_1),\ldots,\Pi_U(y_n)$ satisfy the general position assumption \ref{label_GP1}.
Take some $\eps\in\{\pm 1\}^n$ and $\sigma\in\mathcal{S}_n$. Condition (B1) implies that
$$
(\eps_1y_{\sigma(1)}-\eps_2y_{\sigma(2)})^\perp,\ldots, (\eps_{n-1}y_{\sigma(n-1)}-\eps_ny_{\sigma(n)})^\perp,(\eps_ny_{\sigma(n)})^\perp
$$
are in general position.
Since $U$ is in general position to the arrangement $\mathcal{A}^B(y_1,\ldots,y_n)$, which contains these hyperplanes, the following hyperplanes in $U$
\begin{align*}
U\cap(\eps_1y_{\sigma(1)}-\eps_2y_{\sigma(2)})^\perp,\ldots,U\cap(\eps_{n-1}y_{\sigma(n-1)}-\eps_ny_{\sigma(n)})^\perp,U\cap(\eps_ny_{\sigma(n)})^\perp
\end{align*}
are also in general position in $U$. Since $U\cap (z^\perp)=(\Pi_U(z))^\perp\cap U$ for every $z\in\R^d$, it follows that
\begin{align*}
\big(\eps_1\Pi_U(y_{\sigma(1)})-\eps_2\Pi_U(y_{\sigma(2)})\big)^\perp\hspace*{-1.6pt}\cap U, \ldots,\big(\eps_{n-1}\Pi_U(y_{\sigma(n-1)})-\eps_n\Pi_U(y_{\sigma(n)})\big)^\perp\hspace*{-1.6pt}\cap U, \big(\eps_n\Pi_U(y_{\sigma(n)})\big)^\perp\hspace*{-1.6pt}\cap U
\end{align*}
are in general position in $U$. 
Equivalently, the vectors
\begin{align*}
\eps_1\Pi_U(y_{\sigma(1)})-\eps_2\Pi_U(y_{\sigma(2)}), \ldots,\eps_{n-1}\Pi_U(y_{\sigma(n-1)})-\eps_n\Pi_U(y_{\sigma(n)}),\eps_n\Pi_U(y_{\sigma(n)})
\end{align*}
are in general position in $U$. This means that, under the given assumptions, \ref{label_GP1} is satisfied for $\Pi_U(y_1),\ldots,\Pi_U(y_n)$.

Now, we prove part (i). Let $F_j\in\F_j^B|_U(y_1,\ldots,y_n)$. We can apply Proposition~\ref{Prop_Char_WeylFaces}(i) in the ambient linear subspace $U$ to the projections $\Pi_U(y_1),\ldots,\Pi_U(y_n)$. It follows from this proposition and the representation~\eqref{eq:F_B_U_rep} that there are $1\le l_1<\ldots\le l_{n-k+j}\le n$ and $\eps\in\{\pm 1\}^n$, $\sigma\in\mathcal{S}_n$, such that
\begin{align*}
F_j=F_{\eps,\sigma}^B|_U(l_1,\ldots,l_{n-k+j}) = F_{\eps,\sigma}^B (l_1,\ldots,l_{n-k+j}) \cap U.
\end{align*}
Now, we define $F:=F^B_{\eps,\sigma}(l_1,\ldots,l_{n-k+j})$.  Note that $F\neq\{0\}$ because $F_j\neq \{0\}$.  Since \ref{label_GP1} is satisfied for $y_1,\ldots,y_n$, Proposition~\ref{Prop_Char_WeylFaces}(ii) yields that $F\in\F^B_{d-k+j}(y_1,\ldots,y_n)$. It follows from the construction that $F\cap U=F_j$.

The uniqueness of $F\in\F^B_{d-k+j}(y_1,\ldots,y_n)$ such that  $F\cap U = F_j$ follows from our general position assumptions or rather from the fact that the projections $\Pi_U(y_1),\ldots,\Pi_U(y_n)$ satisfy the assumption \ref{label_GP1} in $U$. We will sketch the idea of the proof. Suppose there is another face $G\in\F^B_{d-k+j}(y_1,\ldots,y_n)$ with $G\cap U =  F_j$. By  Proposition~\ref{Prop_Char_WeylFaces}(i) this means that there are $1\le i_1<\ldots<i_{n-k+j}\le n$ and $\delta\in\{\pm 1\}^n$, $\pi\in\mathcal{S}_n$, such that $G=F^B_{\delta,\,\pi}(i_1,\ldots,i_{n-k+j})$. It follows that
\begin{align*}
F_{\delta,\,\pi}^B(i_1,\ldots,i_{n-k+j})\cap U = F_{\eps,\sigma}^B(l_1,\ldots,l_{n-k+j})\cap U = F_j.
\end{align*}
Consequently,
\begin{align*}
F_{\delta,\,\pi}^B|_U(i_1,\ldots,i_{n-k+j})
=
F_{\eps,\sigma}^B(l_1,\ldots,l_{n-k+j})|_U
= F_j \neq \{0\}.
\end{align*}
Applying Proposition~\ref{Prop uniq} in the ambient space $U$ to the projected vectors $\Pi_U(y_1),\ldots,\Pi_U(y_n)$, we get $\sigma = \pi$, $l_p = i_p$ for all admissible $p$, and $\eps_{p}= \delta_{p}$ for all $1\leq p \leq l_{n-k+j}$.  But this implies that $F=G$, and thus, proves (i).

Now we will prove part (ii). Take $F\in \F^B_{d-k+j}(y_1,\ldots,y_n)$ satisfying $F\cap U\neq\{0\}$. Proposition~\ref{Prop_Char_WeylFaces}(i) implies that there are $1\le l_1<\ldots<l_{n-k+j}\le n$ and $\eps\in\{\pm 1\}^n$, $\sigma\in\mathcal{S}_n$, such that $F=F^B_{\eps,\sigma}(l_1,\ldots,l_{n-k+j})$. As we have seen in~\eqref{eq:F_B_U_rep} above, it follows
\begin{align*}
\{0\}\neq F\cap U=F^B_{\eps,\sigma}|_U(l_1,\ldots,l_{n-k+j}).
\end{align*}
Since $\Pi_U(y_1),\ldots,\Pi_U(y_n)$ satisfy the condition \ref{label_GP1} in the ambient linear subspace $U$, we can apply Proposition~\ref{Prop_Char_WeylFaces}(ii) in the subspace $U$, which yields that $F\cap U$ is $j$-face of the induced Weyl tessellation $\mathcal W^B|_U(y_1,\ldots,y_n)$.
\end{proof}

The $A_{n-1}$-version of this lemma is analogous and would require introducing the respective notation for the representatives of the Weyl faces of type $A_{n-1}$ induced in a linear subspace $U$. Since this finds no further application in this paper, we omit the result.


\subsection{Proofs of Theorems~\ref{Theorem_sizefunction_Dn_An-1} and~\ref{Theorem_sizefunction_Dn}}

Now, we finally prove the formulas for the expected size functionals of the random Weyl cones $\Dnb$:
\begin{align*}
\E Y_{d-k+j,\,d-k}(\Dnb)=\frac{2^{k-j}\tbinom{n}{k-j}D^B(n-k+j,j)}{2D^B(n,d)}\frac{n!}{(n-k+j)!}
\end{align*}
for all $1\le j\le k\le d$. Theorem~\ref{Theorem_sizefunction_Dn_An-1} is proven in the same way, using the respective results for the $A_{n-1}$-case. We omit the proof of the $A_{n-1}$-case.

Recall that the Weyl random cone $\Dnb$ is the cone chosen uniformly at random from the $D^B(n,d)$ cones of the random Weyl tessellation $\mathcal W^B(Y_1,\ldots,Y_n)$, where $Y_1,\ldots,Y_n$ are random vectors in $\R^d$ satsifying~\ref{label_GP1} a.s. Thus, its distribution is given by
\begin{align}
\label{Eq_Distr_Dn}
\PP(\Dnb\in B)=\int_{(\R^d)^n}\frac{1}{D^B(n,d)}\sum_{C\in\mathcal W^B(y_1,\ldots,y_n)}\1_B(C)\,\P_Y(\text{d}(y_1,\ldots,y_n))
\end{align}
for a Borel set $B$ of polyhedral cones, where $\P_Y$ denotes the joint probability law of $(Y_1,\ldots,Y_n)$ on $(\R^d)^n$.

\begin{proof}[Proof of Theorem~\ref{Theorem_sizefunction_Dn}]
Suppose $1\le j\le k\le d$. Using the definition of the size functional and (\ref{Eq_Distr_Dn}), we get
\begin{align*}
\E Y_{d-k+j,\,d-k}(\D^B_n)&=\E\sum_{F\in\F_{d-k+j}(\D_n^B)}U_{d-k}(F)\\
&	=\int_{(\R^{d})^n}\frac{1}{D^B(n,d)}\sum_{C\in\mathcal W^B(y_1,\ldots,y_n)}\:\sum_{F\in\F^B_{d-k+j}(C)}U_{d-k}(F)
\,\P_Y (\text{d}(y_1,\ldots,y_n)).
\end{align*}
In order to apply the definition (\ref{Eq_Quermassintegrals_No_Subspace}) of the quermassintegral $U_{d-k}$ we need to verify that the $(d-k+j)$-faces $F\in \F_{d-k+j}^B(y_1,\ldots,y_n)$ are a.s.\ not linear subspaces.
For this, it suffices to show that any Weyl cone $C\in\mathcal W^B(y_1,\ldots,y_n)$ is pointed (provided that $y_1,\ldots,y_n$ satisfy~\ref{label_GP1}), or equivalently, that $\linsp(C):=(-C)\cap C=\{0\}$. To this end, take a vector $v\in\linsp(C)$. We know that $C$ is of the form
\begin{align*}
\{v\in\R^d:\langle v,\eps_1y_{\sigma(1)}\rangle\le\ldots\le\langle v,\eps_ny_{\sigma(n)}\rangle\le 0\}
\end{align*}
for some $\eps\in\{\pm 1\}^n$ and $\sigma\in\mathcal S_n$.
So if $v\in\linsp(C)$, then $v$ is orthogonal to all of the vectors $\varepsilon_1y_{\sigma(1)}-\varepsilon_2y_{\sigma(2)},\varepsilon_2y_{\sigma(2)}-\varepsilon_3y_{\sigma(3)},\ldots,\varepsilon_{n-1}y_{\sigma(n-1)}-\varepsilon_ny_{\sigma(n)}, \varepsilon_ny_{\sigma(n)}$. Since arbitrary $d$ of these vectors
are linearly independent due to~\ref{label_GP1}, we deduce that $v=0$.

Now, we can apply (\ref{Eq_Quermassintegrals_No_Subspace}) and then interchange the integral and the sums. This yields
\begin{align}\label{Eq_Integral_1}
\E Y_{d-k+j,\,d-k}(\D_n^B)
&	=\frac{1}{2D^B(n,d)}\int_{(\R^{d})^n}\sum_{F\in\F^B_{d-k+j}(y_1,\ldots,y_n)}\:\sum_{C\in\mathcal W^B(y_1,\ldots,y_n)}\1_{\{F\subseteq C\}}\nonumber\\
&	\quad\times\int_{G(d,k)}\1_{\{F\cap U\neq\{ 0\}\}}\,\nu_k(\text{d}U)\,\P_Y (\text{d}(y_1,\ldots,y_n))\nonumber\\
&	=\frac{1}{2D^B(n,d)}\int_{(\R^{d})^n}\int_{G(d,k)}\sum_{F\in\F^B_{d-k+j}(y_1,\ldots,y_n)}\1_{\{F\cap U\neq\{0\}\}}\sum_{C\in\mathcal W^B(y_1,\ldots,y_n)}\1_{\{F\subseteq C\}}\nonumber\\
&	\quad\times\nu_k(\text{d}U)
\,\P_Y (\text{d}(y_1,\ldots,y_n)).
\end{align}
Our goal is to show that the sums inside the integrals are constant for $\nu_k$-almost every $U\in G(d,k)$ and $\P_Y$-almost every $(y_1,\ldots,y_n)\in (\R^d)^n$. Using Lemma~\ref{Lemma_Faces_Induced_in_L}, we obtain
\begin{align}
\sum_{F\in\F^B_{d-k+j}(y_1,\ldots,y_n)}\1_{\{F\cap U\neq\{0\}\}}\sum_{C\in\mathcal W^B(y_1,\ldots,y_n)}\1_{\{F\subseteq C\}}
&=
\sum_{F_j\in\F_j^B|_U(y_1,\ldots,y_n)}\:\sum_{D\in\mathcal W^B|_U(y_1,\ldots,y_n)}\1_{\{F_j\subseteq D\}}\label{Eq_Proof_Size_Dn_Sum}
\end{align}
for almost every $U\in G(d,k)$ and $\P_Y$-almost every $(y_1,\ldots,y_n)$. Indeed, by Lemma~\ref{Lemma_Faces_Induced_in_L} there is  a one-to-one correspondence between the pairs $F\subseteq C$ such that $F\cap U \neq \{0\}$ and the pairs $F_j\subseteq D$ as above.  Note that Lemma~\ref{Lemma_Faces_Induced_in_L} was applicable, since $\nu_k$-almost every $U\in G(d,k)$ is in general position with respect to the arrangement $\mathcal{A}^B(y_1,\ldots,y_n)$, due to Example~\ref{Remark_General_position_Uniform_Subspace}, and almost every set of vectors $(y_1,\ldots,y_n)$ satisfies the general position assumption \ref{label_GP1}.

Applying Proposition~\ref{prop:faces_with_multiplicity_B} to the ambient linear subspace $U$ instead of $\R^d$ and the projections $\Pi_{U}(y_1),\ldots,\Pi_{U}(y_n)$ instead of $y_1,\ldots,y_n$, we obtain
\begin{align}\label{eq:aux1}
\sum_{F_j\in\F_j^B|_U(y_1,\ldots,y_n)}\:\sum_{D\in\mathcal W^B|_U(y_1,\ldots,y_n)}\1_{\{F_j\subseteq D\}}
=
2^{k-j}\binom{n}{k-j}\frac{n!}{(n-k+j)!}D^B(n-k+j,j).
\end{align}
To see that Proposition~\ref{prop:faces_with_multiplicity_B} is applicable, note that $\nu_k$-a.e.\ $U$ is in general position with respect to the arrangement $\A^B(y_1,\ldots,y_n)$ and hence the projections $\Pi_{U}(y_1),\ldots,\Pi_{U}(y_n)$ satisfy assumption~\ref{label_GP1}, as we have shown in the proof of Lemma~\ref{Lemma_Faces_Induced_in_L}.

Inserting~\eqref{eq:aux1} and~\eqref{Eq_Proof_Size_Dn_Sum} into~\eqref{Eq_Integral_1}, we arrive at
$$
\E Y_{d-k+j,\,d-k}(\D_n^B)
=
\frac{1}{2D^B(n,d)} \cdot 2^{k-j}\binom{n}{k-j}\frac{n!}{(n-k+j)!}D^B(n-k+j,j),
$$
which completes the proof.
\end{proof}


\section{General Position: Proofs of Theorems~\ref{Theorem_Aequivalenz_B1_B2} and~\ref{Theorem_Aequivalenz_A1_A2}}\label{Subsection_Proof_GeneralPosition}


\subsection{Equivalences of \ref{label_GP1} and \ref{label_GP2}, \ref{label_GP1_A} and \ref{label_GP2_A}}
In this section, we will prove that assumption~\ref{label_GP1} is equivalent to~\ref{label_GP2} and, similarly, \ref{label_GP1_A} is equivalent to~\ref{label_GP2_A}.

\begin{proof}[Proof of Theorem~\ref{Theorem_Aequivalenz_B1_B2}]\label{Proof_Aequivalenz_B1_B2}
Take some vectors $y_1,\ldots,y_n\in\R^d$ with $n\ge d$. We claim that the following conditions are equivalent:
\begin{enumerate}[label=(B\arabic*), leftmargin=50pt]
\item For every $\varepsilon=(\varepsilon_1,\ldots,\varepsilon_n)\in\{\pm 1\}^n$ and $\sigma\in\mathcal{S}_n$ the vectors $\varepsilon_1y_{\sigma(1)}-\varepsilon_2y_{\sigma(2)},\varepsilon_2y_{\sigma(2)}-\varepsilon_3y_{\sigma(3)},\ldots,\varepsilon_{n-1}y_{\sigma(n-1)}-\varepsilon_ny_{\sigma(n)}, \varepsilon_ny_{\sigma(n)}$ are in general position.
\item The linear subspace $L^\perp$ has dimension $d$ and is in general position with respect to the reflection arrangement $\mathcal{A}(B_n)$, where $L:=\{\beta\in\R^n:\beta_1y_1+\ldots+\beta_ny_n=0\}$.
\end{enumerate}

At first, we prove that \ref{label_GP2} implies \ref{label_GP1}. Assume that~\ref{label_GP2} holds true but at the same time~\ref{label_GP1} is violated. Then, there exist $\eps\in\{\pm 1\}^n$ and $\sigma\in\mathcal{S}_n$ such that the vectors
\begin{align*}
\eps_1y_{\sigma(1)}-\eps_2y_{\sigma(2)},\ldots,\eps_{n-1}y_{\sigma(n-1)}-\eps_ny_{\sigma(n)},\eps_ny_{\sigma(n)}
\end{align*}
are not in general position. Applying a suitable signed permutation of the coordinates, we may assume that $\eps_i=1$ and $\sigma(i)=i$ for all $i$.  Thus, $y_1-y_2,\ldots,y_{n-1}-y_n,y_n$ are not in general position. This means that there is a subset of $d$ or fewer linearly dependent vectors. In general, this subset is of the form
\begin{align*}
\underbrace{y_1-y_2,\ldots,y_{i_1-1}-y_{i_1}}_{\text{group $1$}},\underbrace{y_{i_1+1}-y_{i_1+2},\ldots,y_{i_2-1}-y_{i_2}}_{\text{group $2$}},\ldots,\underbrace{y_{i_k+1}-y_{i_k+2},\ldots,y_{n-1}-y_n,y_n}_{\text{group $k+1$}}
\end{align*}
for a $k\ge n-d$ and suitable indices $1\le i_1<i_2<\ldots<i_k\le n$. Note that each of these groups may be empty and the set consists of $n-k\le d$ vectors. This set is linearly dependent if and only if there exist numbers $\lambda_i$ with $i\in\{1,\ldots,n\}\backslash\{i_1,\ldots,i_k\}$ that do not vanish simultaneously and such that
\begin{align}
0
&	=\lambda_1(y_1-y_2)+\ldots+\lambda_{i_1-1}(y_{i_1-1}-y_{i_1})+\lambda_{i_1+1}(y_{i_1+1}-y_{i_1+2})+\ldots+\lambda_{i_2-1}(y_{i_2-1}-y_{i_2})\notag\\	&	\quad+\ldots+\lambda_{i_k+1}(y_{i_k+1}-y_{i_k+2})+\ldots+\lambda_{n-1}(y_{n-1}-y_n)+\lambda_ny_n. \label{eq:aaaaa}
\end{align}
Denote by $e_1,\ldots,e_n$ the standard Euclidean basis in $\R^n$ and define a linear operator $A:\R^n \to \R^d$ by $A e_1 = y_1,\ldots, A e_n = y_n$.
%
Then, the above equality holds if and only if there exist numbers $\lambda_i$ with $i\in\{1,\ldots,n\}\backslash\{i_1,\ldots,i_k\}$ that do not vanish simultaneously and such that the vector
\begin{multline*}
\lambda_1(e_1-e_2)+\ldots+\lambda_{i_1-1}(e_{i_1-1}-e_{i_1})+\lambda_{i_1+1}(e_{i_1+1}-e_{i_1+2})+\ldots+\lambda_{i_2-1}(e_{i_2-1}-e_{i_2})+\\
\ldots+\lambda_{i_k+1}(e_{i_k+1}-e_{i_k+2})+\ldots+\lambda_{n-1}(e_{n-1}-e_n)+\lambda_ne_n
\end{multline*}
is contained in $\Ker A =\{\beta\in\R^n:\beta_1y_1+\ldots+\beta_ny_n=0\}=L$. This is equivalent to
\begin{multline*}
\lin\big\{e_1-e_2,\ldots,e_{i_1-1}-e_{i_1},e_{i_1+1}-e_{i_1+2},\ldots,e_{i_2-1}-e_{i_2},\\\ldots,e_{i_k+1}-e_{i_k+2},\ldots,e_{n-1}-e_n,e_n\big\}\cap L\neq \{0\}.
\end{multline*}
This holds if and only if $K^\perp\cap L\neq \{0\}$ for the $k$-dimensional subspace
\begin{align*}
K
	=\{\beta\in\R^n:\beta_1=\ldots=\beta_{i_1},\ldots,\beta_{i_{k-1}+1}=\ldots=\beta_{i_k},\beta_{i_k+1}=\ldots=\beta_n=0\}. 
\end{align*}
We observe that $K$ is an intersection of hyperplanes from the reflection arrangement $\A(B_n)$. Then, $K^\perp\cap L\neq \{0\}$ is equivalent to
\begin{align*}
\dim(L^\perp\cap K)=n-\dim(L+K^\perp)=d-n+k+\dim(L\cap K^\perp) \neq d-n+k,
\end{align*}
since $\dim(L)=n-d$. This means that $L^\perp$ is not in general position to $\A(B_n)$, which is a contradiction to \ref{label_GP2}.


It is left to prove that \ref{label_GP1} implies \ref{label_GP2}. Let \ref{label_GP1} hold true for $y_1,\ldots,y_n\in\R^d$. This implies that $\dim L^\perp=d$. In order to prove this, it is enough to show that, for example, the set of $d$ vectors $y_{n-d+1},\ldots,y_n$ is linearly independent. Suppose $\lambda_{n-d+1}y_{n-d+1}+\ldots+\lambda_ny_n=0$ holds for some $\lambda_{n-d+1},\ldots,\lambda_n\in\R$. Representing the individual $y_j$'s as telescope sums, this implies
\begin{align*}
0&=\lambda_{n-d+1}((y_{n-d+1}-y_{n-d+2})+\ldots+(y_{n-1}-y_n)+y_n)+\ldots+\lambda_{n-1}((y_{n-1}-y_n)+y_n)+\lambda_ny_n\\
&	=\lambda_{n-d+1}(y_{n-d+1}-y_{n-d+2})+\ldots+(\lambda_{n-d+1}+\ldots+\lambda_n)y_n.
\end{align*}
Since $y_{n-d+1}-y_{n-d+2},\ldots,y_{n-1}-y_n,y_n$ are linearly independent, due to \ref{label_GP1}, it follows that $\lambda_{n-d+1}=\ldots=\lambda_n=0$, which proves the linear independence of $y_{n-d+1},\ldots,y_n$.

Now, suppose $L^\perp$ is not in general position to $\A(B_n)$. Thus, there exists a $k$-dimensional subspace $K'$ that can be represented as the intersection of hyperplanes from $\A(B_n)$, such that
\begin{align*}
\dim(K'\cap L^\perp)
\neq\max\{0,d-n+k\}.
\end{align*}
The linear subspace $K'$ is given by a set of equations of the following form. The coordinates $\beta_1,\ldots,\beta_n$ are decomposed into $k+1$ distinguishable groups. These groups are required to be non-empty except the last one. All coordinates in the last group must be 0. For the remaining variables there is a unique choice of signs, which multiplies each variable by $+1$ or $-1$, such that the sign-changed variables are equal inside every group, except the last one.
Applying a suitable signed permutation of the coordinates,  we may assume that $K'$ has the following form:
\begin{align}\label{Eq_Subspace_L(B_n)_Bsp}
K' = K = \{\beta\in\R^n:\beta_1=\ldots=\beta_{i_1},\ldots,\beta_{i_{k-1}+1}=\ldots=\beta_{i_k},\beta_{i_k+1}=\ldots=\beta_n=0\}
\end{align}
for some $1\le i_1<\ldots<i_k\le n$. 

At first, suppose $k\ge n-d$, which implies $\dim(K'\cap L^\perp)\neq d-n+k$.
Now we can refer to the first part of the proof, since all the steps in the argument are equivalences. We conclude that
$y_1-y_2,\ldots,y_{n-1}-y_n,y_n$ are not in general position, and thus, \ref{label_GP1} is not satisfied. 

In the case $k <  n-d$, we know that $\dim(K'\cap L^\perp)>0$.
Thus, there is a linear subspace $K''\supseteq K'$ that can also be represented as the intersection of hyperplanes from $\A(B_n)$, such that $\dim(K'')=n-d$ and $\dim(K''\cap L^\perp)>0$. Note that this subspace $K''$ can be obtained by deleting $(n-d)-k$ equations in the defining condition of $K'$. The previous case, applied to $K''$ instead of $K'$, yields that  the general position assumption \ref{label_GP1} is not satisfied, which is a contradiction.
\end{proof}


\begin{proof}[Proof of Theorem~\ref{Theorem_Aequivalenz_A1_A2}]
Similar to the proof of Theorem~\ref{Theorem_Aequivalenz_B1_B2}.
\end{proof}
\subsection{Sufficient conditions for general position}
Let us state a simple yet general sufficient condition under which the
assumptions~\ref{label_GP1_A} and~\ref{label_GP1} are fulfilled a.s.
\begin{lem}\label{Lemma_AS_General_Position}
Let $\mu$ be a $\sigma$-finite Borel measure on $\R^d$  that assigns measure zero to each affine hyperplane,
i.e.~each $(d-1)$-dimensional affine subspace.
Furthermore, let $Y_1,\ldots,Y_n$ be random vectors in $\R^d$ having a joint $\mu^n$-density on $(\R^d)^n$.
Then  $Y_1,\ldots,Y_n$ satisfy assumptions~\ref{label_GP1_A}
(provided $n\geq d+1$) and~\ref{label_GP1} (provided $n\geq d$) with probability $1$.
\end{lem}

\begin{proof}
Since \ref{label_GP1} implies~\ref{label_GP1_A}, we only need to prove that \ref{label_GP1} holds a.s.
Since $Y_1,\ldots,Y_n$ have a joint density function with respect to $\mu^n$, so do $\eps_1Y_{\sigma(1)},\ldots,\eps_nY_{\sigma(n)}$, for each $\eps\in\{\pm 1\}^n$ and $\sigma\in \mathcal{S}_n$. Therefore, it suffices to prove that $Y_1-Y_2,\ldots,Y_{n-1}-Y_n,Y_n$ are in general position a.s., or equivalently, that  they are not in general position with probability $0$. To this end,
suppose there is a subset of $n-k\leq d$ linearly dependent vectors. Recalling the proof of Theorem~\ref{Theorem_Aequivalenz_B1_B2}, this set is of the form
\begin{align*}
Y_1-Y_2,\ldots,Y_{i_1-1}-Y_{i_1},Y_{i_1+1}-Y_{i_1+2},\ldots,Y_{i_2-1}-Y_{i_2},\ldots,Y_{i_k+1}-Y_{i_k+2},\ldots,Y_{n-1}-Y_n,Y_n
\end{align*}
for suitable indices $1\le i_1<\ldots<i_k\le n$. This means that there are numbers $\lambda_i$ with $i\in\{1,\ldots,n\}\backslash\{i_1,\ldots,i_k\}$ that do not vanish simultaneously and such that
\begin{align}\label{Eq_General_Pos_Aff_Subspace}
\lambda_1Y_1	
&	=(\lambda_1-\lambda_2)Y_2+\ldots+(\lambda_{i_1-2}-\lambda_{i_1-1})Y_{i_1-1}+\lambda_{i_1-1}Y_{i_1}\nonumber\\
&	\quad+(-\lambda_{i_1+1})Y_{i_1+1}+(\lambda_{i_1+1}-\lambda_{i_1+2})Y_{i_1+2}+\ldots+(\lambda_{i_2-2}-\lambda_{i_2-1})Y_{i_2-1}+\lambda_{i_2-1}Y_{i_2}\\
&	\quad+\ldots+(-\lambda_{i_k+1})Y_{i_k+1}+(\lambda_{i_k+1}-\lambda_{i_k+2})Y_{i_k+2}+\ldots+(\lambda_{n-1}-\lambda_n)Y_n\nonumber
\end{align}
holds true (see \eqref{eq:aaaaa} solved for $\lambda_1Y_1$). Without loss of generality, we may assume that $\lambda_1\neq 0$ (otherwise, choose the smallest $i$, such that $\lambda_i\neq 0$ and solve for $\lambda_iY_i$). Divide~\eqref{Eq_General_Pos_Aff_Subspace} by $\lambda_1$.   The possible values of the first line coincide with the affine hull of $Y_2,\ldots,Y_{i_1}$ denoted by $\aff\{Y_2,\ldots,Y_{i_1}\}$, since the coefficients of the $Y_i$'s satisfy the relation
\begin{align*}
\frac{\lambda_1-\lambda_2}{\lambda_1}+\ldots+\frac{\lambda_{i_1-2}-\lambda_{i_1-1}}{\lambda_1}+\frac{\lambda_{i_1-1}}{\lambda_1}=1.
\end{align*}
The dimension of this affine subspace is at most $i_1-2$. The possible values of the second line of (\ref{Eq_General_Pos_Aff_Subspace}), divided by $\lambda_1$,  define the linear subspace
\begin{align*}
L_1:=\{\beta_{i_1+1}Y_{i_1+1}+\ldots+\beta_{i_2}Y_{i_2}:\beta_{i_1+1}+\ldots+\beta_{i_2}=0\},
\end{align*}
since the coefficients satisfy the relation
\begin{align*}
\frac{-\lambda_{i_1+1}}{\lambda_1}+\frac{\lambda_{i_1+1}-\lambda_{i_1+2}}{\lambda_1}+\ldots+\frac{\lambda_{i_2-2}-\lambda_{i_2-1}}{\lambda_1}+\frac{\lambda_{i_2-1}}{\lambda_1}=0.
\end{align*}
Similarly, the subsequent lines, except the last one, define linear subspaces $L_2,\ldots,L_{k-1}$. The dimension of the linear subspaces $L_1,\ldots,L_{k-1}$ is at most $i_2-i_1-1,\ldots,i_{k}-i_{k-1}-1$, respectively. Thus, (\ref{Eq_General_Pos_Aff_Subspace}) implies that
\begin{align*}
Y_1\in L(Y_2,\ldots,Y_n):=\aff\{Y_2,\ldots,Y_{i_1}\}+L_1+\ldots+L_{k-1}+\lin\{Y_{i_k+1},\ldots,Y_n\},
\end{align*}
and the dimension of the affine subspace $L(Y_2,\ldots,Y_n)$ is at most
\begin{align*}
(i_1-2)+(i_2-i_1-1)+\ldots+(i_k-i_{k-1}-1)+(n-i_k)=n-k-1<d.
\end{align*}

It remains to show that the event $Y_1\in L(Y_2,\ldots,Y_n)$ has probability $0$. Now, since $(Y_1,\ldots,Y_n)$ has a joint $\mu^n$-density, the conditional $\mu$-density of $Y_1$ conditioned on the event that $(Y_2,\ldots,Y_n)=(y_2,\ldots,y_n)$ exists.
Recalling that  $\mu$ assigns measure $0$ to each affine hyperplane, we conclude that  the conditional probability that $Y_1\in L(Y_2,\ldots,Y_n)$  given that $(Y_2,\ldots,Y_n)=(y_2,\ldots,y_n)$ vanishes.  Integrating over all tuples $(y_2,\ldots,y_n)$, we conclude that the probability that $Y_1\in L(Y_2,\ldots,Y_n)$ is $0$.
\end{proof}



\section*{Acknowledgement}
Supported by the German Research Foundation under Germany's Excellence Strategy  EXC 2044 -- 390685587, \textit{Mathematics M\"unster: Dynamics - Geometry - Structure}  and by the DFG priority program SPP 2265 \textit{Random Geometric Systems}.
We thank R.\ Schneider for pointing out an error related to Lemma~\ref{Lemma_Cone_Subspace_LinSp} and to the anonymous referees for useful suggestions that considerably improved the presentation.

\vspace{1cm}

\bibliography{bibliography}
\bibliographystyle{abbrv}



\end{document}